\documentclass[journal]{IEEEtran}

\usepackage{amsmath,amsfonts,amssymb,amsthm,graphicx}
\usepackage{color}
\usepackage{enumerate}
\usepackage[ruled,vlined]{algorithm2e}
\usepackage{soul}
\usepackage[normalem]{ulem}
\usepackage{ifthen}
\usepackage[colorinlistoftodos]{todonotes}
\usepackage{citesort}
\usepackage{bm}

\newtheorem{proposition}{Proposition}
\newtheorem{definition}{Definition}
\newtheorem{theorem}{Theorem}

\newcommand{\bs}{\boldsymbol{s}}
\newcommand{\blambda}{\boldsymbol{\lambda}}
\newcommand{\bc}{\boldsymbol{c}}
\newcommand{\bq}{\boldsymbol{q}}
\newcommand{\by}{\boldsymbol{y}}
\newcommand{\bv}{\boldsymbol{v}}
\newcommand{\bz}{\boldsymbol{z}}
\newcommand{\bx}{\boldsymbol{x}}
\newcommand{\bS}{\boldsymbol{\mathcal{S}}}
\newcommand{\bA}{\boldsymbol{A}}
\newcommand{\bcS}{\boldsymbol{\mathcal{S}}}
\newcommand{\bcA}{\boldsymbol{\mathcal{A}}}
\newcommand{\bM}{\mathbf{M}}
\newcommand{\bB}{\boldsymbol{B}}

\def\argmin{\mathop{\mathrm{arg\,min}}} 
\newcommand{\prox}{\textbf{prox}}
\newcommand{\R}{\mathbb{R}}
\newcommand{\N}{\mathbb{N}}
\newcommand{\C}{\mathbb{C}}
\newcommand{\D}{\dot{\bM}}
\newcommand{\DD}{\ddot{\bM}}

\newcommand{\dd}{\mathrm{d}}
\newcommand{\indic}{\imath}
\newif\ifdraftvar
\draftvartrue
\ifdraftvar

\else

\fi

\begin{document}
\graphicspath{{figures/}}
\newcommand{\colvec}[2][.8]{%
  \scalebox{#1}{$\begin{array}{@{}c@{}}#2\end{array}$}}%

\title{Gradient waveform design for variable density sampling in Magnetic Resonance Imaging}

\author{Nicolas~Chauffert,~\IEEEmembership{}
        Pierre~Weiss,~\IEEEmembership{}
        Jonas~Kahn~\IEEEmembership{}
        and~Philippe~Ciuciu.~\IEEEmembership{}
\thanks{N. Chauffert and P. Ciuciu are with Inria Parietal Team / NeuroSpin center, CEA Saclay. Contact: nicolas.chauffert@gmail.com}
\thanks{P. Weiss is with PRIMO Team, ITAV, USR 3505, Universit\'e de Toulouse.}
\thanks{J. Kahn is with Laboratoire Painlev\'e, UMR 8524, Universit\'e de Lille 1, CNRS.}
}

%

\maketitle

\begin{abstract}
Fast coverage of $k$-space is a major concern to speed up data acquisition in Magnetic Resonance Imaging~(MRI) and limit image distortions due to long echo train durations. The hardware gradient constraints~(magnitude, slew rate) must be taken into account to collect a sufficient amount of samples in a minimal amount of time. However, sampling strategies (e.g., Compressed Sensing) and optimal gradient waveform design have been developed separately so far. The major flaw of existing methods is that they do not take the sampling density into account, the latter being central in sampling theory. In particular, methods using optimal control tend to agglutinate samples in high curvature areas. 
In this paper, we develop an iterative algorithm to project any parameterization of $k$-space trajectories onto the set of feasible curves that fulfills the gradient constraints. We show that our projection algorithm provides a more efficient alternative than existinf approaches and that it can be a way of reducing acquisition time while maintaining sampling density for piece-wise linear trajectories.
\end{abstract}

\begin{IEEEkeywords}
gradient waveform design, $k$-space trajectories, variable density sampling, gradient hardware constraints, magnetic resonance imaging
\end{IEEEkeywords}

\IEEEpeerreviewmaketitle

\section{Introduction}
\IEEEPARstart{T}{he} advent of new hardware and sampling theories provide unprecedented opportunities to reduce acquisition times in MRI.
The design of gradient waveforms minimizing the acquisition time while providing enough information to reconstruct distortion-free images is however an important challenge.
Ideally, these two concerns (sampling scheme and gradient waveform designs) should be addressed simultaneously but it is unclear how to design a feasible waveform corresponding to a $k$-space sampling with good space coverage properties. The gradient waveform design is thus generally performed sequentially: a first step aims at finding the trajectory support or at least control points, and a second step consists of designing the fastest gradient waveforms to traverse the trajectory.

Recent progresses in compressed sensing~(CS) \cite{Puy11,Chauffert14,Krahmer12,Adcock13} give hints concerning the best strategies of $k$-space sampling. In particular, the $k$-space has to be sampled with a variable density, decaying from the lowest to the highest frequencies. To date, CS sampling schemes are limited to simple trajectories, such as 2D independent sampling and acquisition in the orthogonal readout direction. The design of physically plausible acquisition schemes with a prescribed sampling density remains an open question. To better understand how the proposed approach goes beyond the state-of-the-art, we start with a short review of existing methods for designing gradient waveforms for k-space sampling. We show that existing methods do not encompass the case of curves with high curvature in terms of sampling density and scanning time, justifying a new strategy for designing gradient waveforms.

\subsection{Related works}
Some heuristic methods have been proposed to design simultaneously the sampling density and the gradient waveforms~\cite{Mir04,Dale04} but these are not computationally efficient and it is not clear how to sample the $k$-space according to a prescribed distribution using these methods. Let us mention other algorithms that have been proposed to traverse specific curve shapes, such as spiral~\cite{gurney06,pipe14} or concentric rings~\cite{wu08}.


It is in general easier to consider a sampling trajectory in a first time and design the gradient waveform afterwards. To the best of our knowledge, the latter problem is currently solved using convex optimization~\cite{simonetti93,Hargreaves04}, optimal control~\cite{Lustig08}, or optimal interpolation of $k$-space control points~\cite{Davids14b}. Given an input curve, both methods find a parameterization that minimizes the time required to traverse the curve subject to kinematic constraints. This principle suffers from two limitations in certain situations. First, reparameterizing the curve changes the density of samples along the curve. This density is now known to be a key factor in compressed sensing~\cite{Puy11,Chauffert14,Krahmer12,Adcock13}, since it directly impacts the number of required measurements to ensure accurate reconstruction. Second, the challenge of rapid acquisitions is to reduce the scanning time~(echo train duration) and limit geometric distortions induced by inhomogeneities of the static magnetic field~($B_0$) by covering the $k$-space as fast as possible. The perfect fit to an arbitrary curve may be time consuming, especially in the high curvature parts of the trajectory. In particular, the time to traverse piecewise linear sampling trajectories~\cite{Chauffert14,Chauffert13b,Chauffert13c,Wang12,Willett} may become too long. Indeed, the magnetic field gradients have to be set to zero at each singular point of such trajectory. For this reason, new methods have to be pushed forward to fill this gap.

\subsection{Contributions}
In this paper, we propose an alternative method based on a convex optimization formulation.
Given any \emph{parameterized} curve, our algorithm returns the \emph{closest} curve that fulfills the gradient constraints.
The main advantages of the proposed approach are the following: i) the time to traverse the $k$-space is usually shorter, ii) the distance between the input and output curves is the quantity to be minimized ensuring a low deviation to the original sampling distribution, iii) the maximal acquisition time may be fixed, enabling to find the closest curve in a given time and iv) it is flexible enough to handle additional hardware constraints (e.g., trajectory starting from the $k$-space center) in the same framework. 

\subsection{Paper organization}
In Section~\ref{sec:design}, we review the formulation of MRI acquisition, by recalling the gradient constraints and introducing the projection problem. Then, in Section~\ref{sec:VDS}, it is shown that curves generated by the proposed strategy (initial parameterization + projection onto the set of physical constraints) may be used to design MRI sampling schemes with locally variable densities. In Section~\ref{sec:resol}, we provide an optimization algorithm to solve the projection problem, and estimate its rate of convergence. Finally, we illustrate the behaviour of our algorithm on two particular cases: rosette and Travelling Salesman Problem-based trajectories~(Section~\ref{sec:illustr}).

\section{Design of $k$-space trajectories using physical gradient waveforms.}
\label{sec:design}
In this section, we recall classical modelling of the acquisition constraints in MRI~\cite{Hargreaves04,Lustig08}. We justify the lack of accuracy of current reparameterization methods in the context of variable density sampling, and motivate the introduction of a new \emph{projection} algorithm.

\subsection{Sampling in MRI}

In MRI, images are sampled in the $k$-space domain along parameterized curves $s : [0, T] \mapsto \R^d $ where $d\in \{2,3\}$ denotes the image dimensions.
The $i$-th coordinate of $s$ is denoted $s_i$.
Let $u:\R^d\to \C$ denote a $d$ dimensional image and $\hat u$ be its Fourier transform. 
Given an image $u$, a curve $s:[0,T]\to \R^d$ and a sampling step $\Delta t$, the image $u$ shall be reconstructed using the set\footnote{For ease of presentation, we assume that the values of $u$ in the $k$-space correspond to its Fourier transform and we neglect distorsions occuring in MRI such as noise.}: 
\begin{equation}\label{eq:samplingset}
\mathcal{E}=\left\{\hat u(s(j\Delta t)), 0\leqslant j \leqslant \left\lfloor \frac{T}{\Delta t} \right\rfloor\right\}.
\end{equation}

\subsection{Gradient constraints}

The gradient waveform associated with a curve $s$ is defined by $ g(t)=\gamma^{-1} \dot{s} (t)$, where $\gamma$ denotes the gyro-magnetic ratio~\cite{Hargreaves04}.
The gradient waveform is obtained by energizing gradient coils (arrangements of wire) with electric currents.

\subsubsection{kinematic constraints}
due to obvious physical constraints, these electric currents have a bounded amplitude and cannot vary too rapidly~(slew rate). Mathematically, these constraints read: 
\begin{align*}
\|g\| \leqslant G_{\max} \qquad \mbox{and} \qquad \|\dot g\| \leqslant S_{\max}
\end{align*}
where $\|\cdot\|$ denotes either the $\ell^\infty$-norm defined by $\|f\|_\infty:=\max_{1\leq i \leq d}\sup_{t\in [0,T]} |f_i(t)|$, or the $\ell^{\infty,2}$-norm defined by $\|f\|_{\infty,2}:=\sup_{t\in [0,T]} \bigl(\sum_{i=1}^d |f_i(t)|^2\bigr)^{\frac{1}{2}}$.
These constraints might be Rotation Invariant~(RIV) if $\|\cdot\| = \|\cdot\|_{\infty,2}$ or Rotation Variant~(RV) if $\|\cdot\| = \|\cdot\|_{\infty}$, depending on whether each gradient coil is energized independently from others or not.
The set of kinematic constraints is denoted $\mathcal{S}$:
\begin{align}
\mathcal{S}:=\left\{s\in \left(\mathcal{C}^2([0,T])\right)^d, \|\dot s\| \leqslant \alpha , \|\ddot s\| \leqslant \beta\right\}.
\end{align}

\subsubsection{Additional affine constraints}
Specific MRI acquisitions may require additional constraints, such as:
\begin{itemize}
\item
Imposing that the trajectory starts from the $k$-space center (i.e., $s(0)=0$) to save time and avoid blips. 
The end-point can also be specified by $s(T)=s_T$, if $s_T$ can be reached during travel time $T$.

\item In the context of multi-shot MRI acquisition, several radio-frequency pulses are necessary to cover the whole $k$-space. Hence, it makes sense to enforce the trajectory to start from the $k$-space center at every $TR$ (repetition time):
$s(k \cdot TR)=0 , 0\leqslant k\leqslant \left\lfloor\frac{T}{TR}\right\rfloor $.

\item In addition to starting from the $k$-space center, one could impose starting at speed $0$, i.e. $\dot s(0)=0$.

\item To avoid artifacts due to flow motion in the object of interest, gradient moment nulling~(GMN) techniques have been introduced in~\cite{majewski10}. In terms of constraints, nulling the $i^{\text{th}}$ moment reads $\int_t t^i g(t) \dd t=0$.
\end{itemize}

Each of these constraints can be modeled by an affine relationship. Hereafter, the set of affine constraints is denoted by $\mathcal{A}$:
$$\mathcal{A}:=\left\{s: [0,T]\to \R^d, A(s)=v \right\},$$
where $v$ is a vector of parameters in $\R^p$ ($p$ is the number of additional constraints) and $A$ is a linear mapping from the curves space to $\R^p$.

A sampling trajectory $s:[0,T]\to \R^d$ will be said to be \emph{admissible} if it belongs to the set $\mathcal{S}\cap\mathcal{A}$. In what follows, we assume that this set is non-empty, i.e. $\mathcal{S}\cap\mathcal{A}\neq \emptyset$. Moreover, we assume, without loss of generality, that the linear constraints are independent (otherwise some could be removed). 

\subsection{Finding an optimal reparameterization}

The traditional approach to design an admissible curve $s\in \mathcal{S}$ given an arbitrary curve $c:[0,T]\to\R^d$ consists of finding a reparameterization $p$ such that $s=c\circ p$ satisfies the physical constraints while minimizing the acquisition time. This problem can be cast as follows:
\begin{align}\label{eq:reparam}
T_{\texttt{Rep}}=\min T'  \text{ such that }  \exists\, p:[0,T'] \mapsto [0,T],\; c \circ p \in \mathcal{S}.
\end{align}
It can be solved efficiently using optimal control~\cite{Lustig08} or convex optimization~\cite{Hargreaves04}.	
The resulting solution $s=c\circ p$ has the same support as $c$. 
This method however suffers from an important drawback when used in the CS framework: it does not provide any control on the density of samples along the curve. For example, for a given curve support shown in Fig.~\ref{fig:illustr_proj}(a), we illustrate the new parameterization (keeping the same support) and the corresponding magnetic field gradients (see Fig.~\ref{fig:illustr_proj}(b) for a discretization of the curve and (c) for the gradient). We notice that the new parameterized curve has to stop at every angular point of the trajectory, yielding more time spent by the curve in the neighbourhood of these points (and more points in the discretization of the curve in Fig.~\ref{fig:illustr_proj}(b)). This phenomenon is likely to modify the sampling distribution, as illustrated in Section~\ref{sec:VDS}. 


The next part is dedicated to introducing an alternative method relaxing the constraint of keeping the same support as $c$.

\subsection{Projection onto the set of constraints}

The idea we propose in this paper is to find the projection of the given input curve $c$ onto the set of admissible curves $\mathcal{S}$:
\begin{align}
\label{pb:primal}
s^*:=\underset{s\in \mathcal{S}}{\operatorname{argmin}} \frac{1}{2} d^2(s,c) = \underset{s\in \mathcal{S}}{\operatorname{argmin}} \frac{1}{2} \|s-c\|_2^2
\end{align}
where $d^2(s,c)=\|s-c\|_2^2:=\int_{t=0}^T \|s(t)-c(t)\|_2^2 \,dt$. 
This method presents important differences compared to the optimal control approach introduced here above: i) the solution $s^*$ and $c$ have different support~(see Fig.~\ref{fig:illustr_proj}(d)) unless $c$ is admissible; ii) the sets composed of the discretization of $c$ and $s^*$ at a given sampling rate are close to each other (Fig.~\ref{fig:illustr_proj}(e)); iii) the acquisition time $T$ is fixed and equal to that of the input curve $c$. In particular, time to traverse a curve is in general shorter than using exact parameterization~(see Fig.~\ref{fig:illustr_proj}(f) where $T<T_{\texttt{Rep}}$). 

\begin{figure}[!ht]
\begin{tabular}{cccc}
(a)&(b)& &(c)\\[-.01 \linewidth]
\includegraphics[width=.26\linewidth]{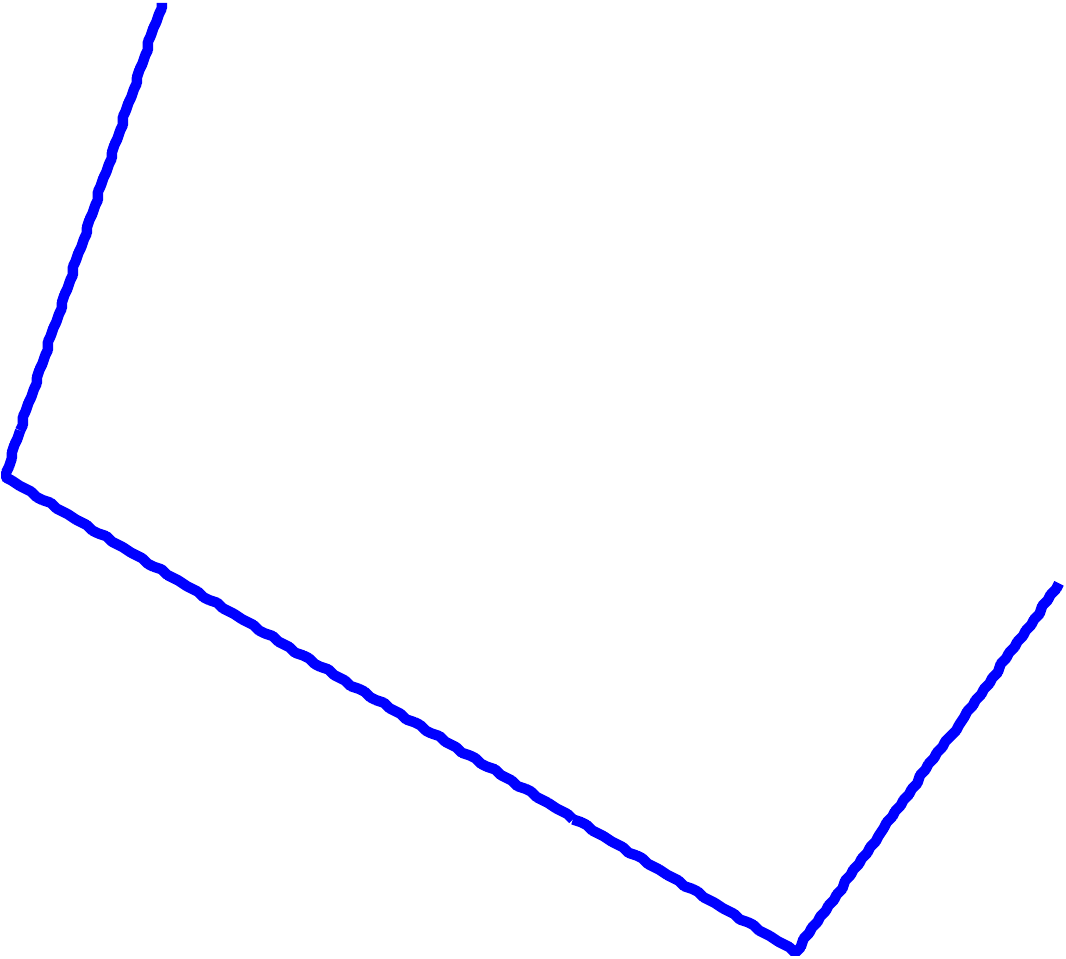}&
\includegraphics[width=.26\linewidth]{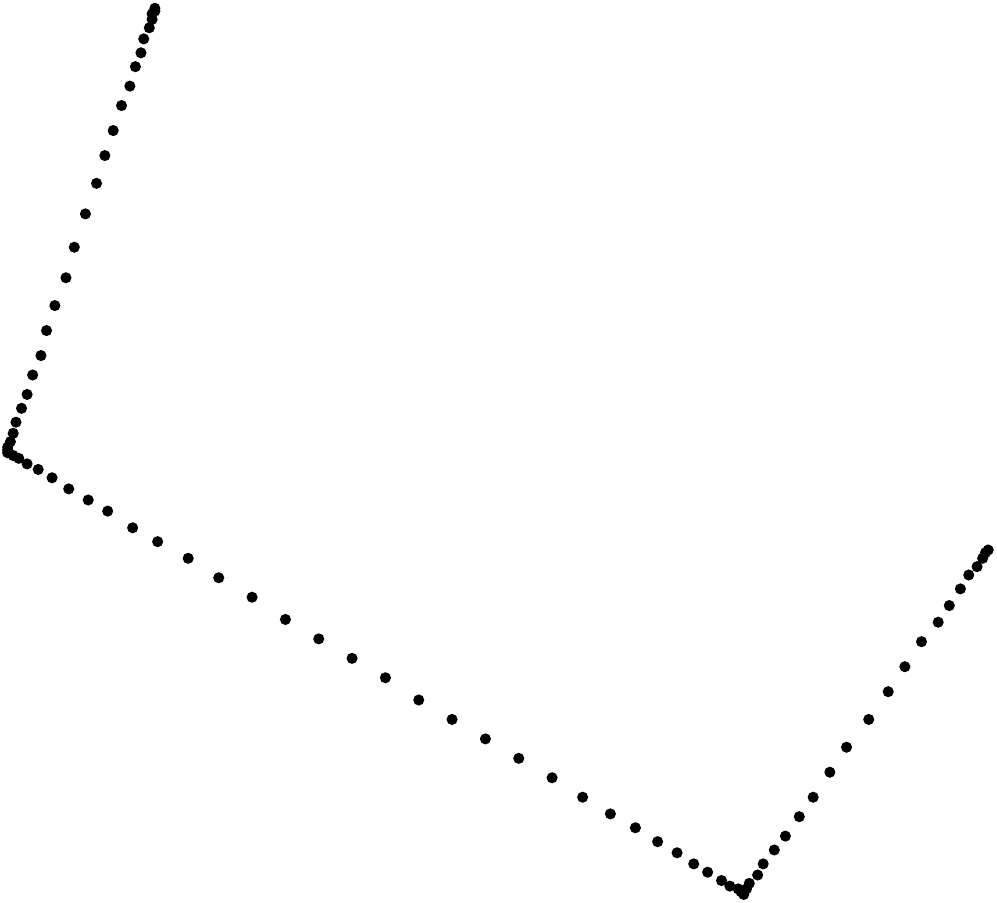}&
\hspace{-.03\linewidth}
\rotatebox{90}{\hspace{0.08\linewidth} $g(t)$} &
\hspace{-.052\linewidth}
\includegraphics[width=.33\linewidth]{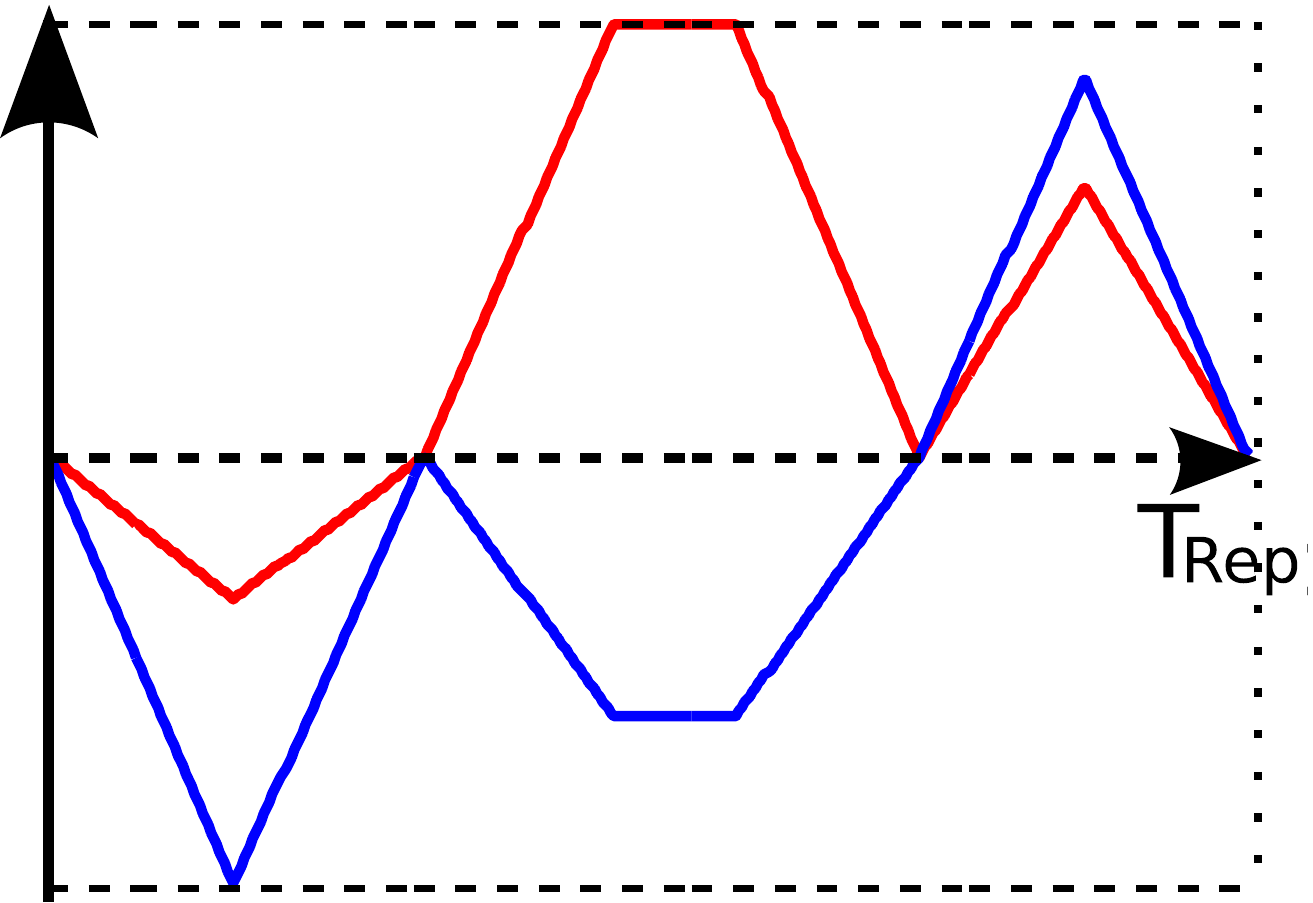}\\[-.01 \linewidth]
 & & & $t$ \\
(d)&(e)& &(f)\\[-.01 \linewidth]
\includegraphics[width=.26\linewidth]{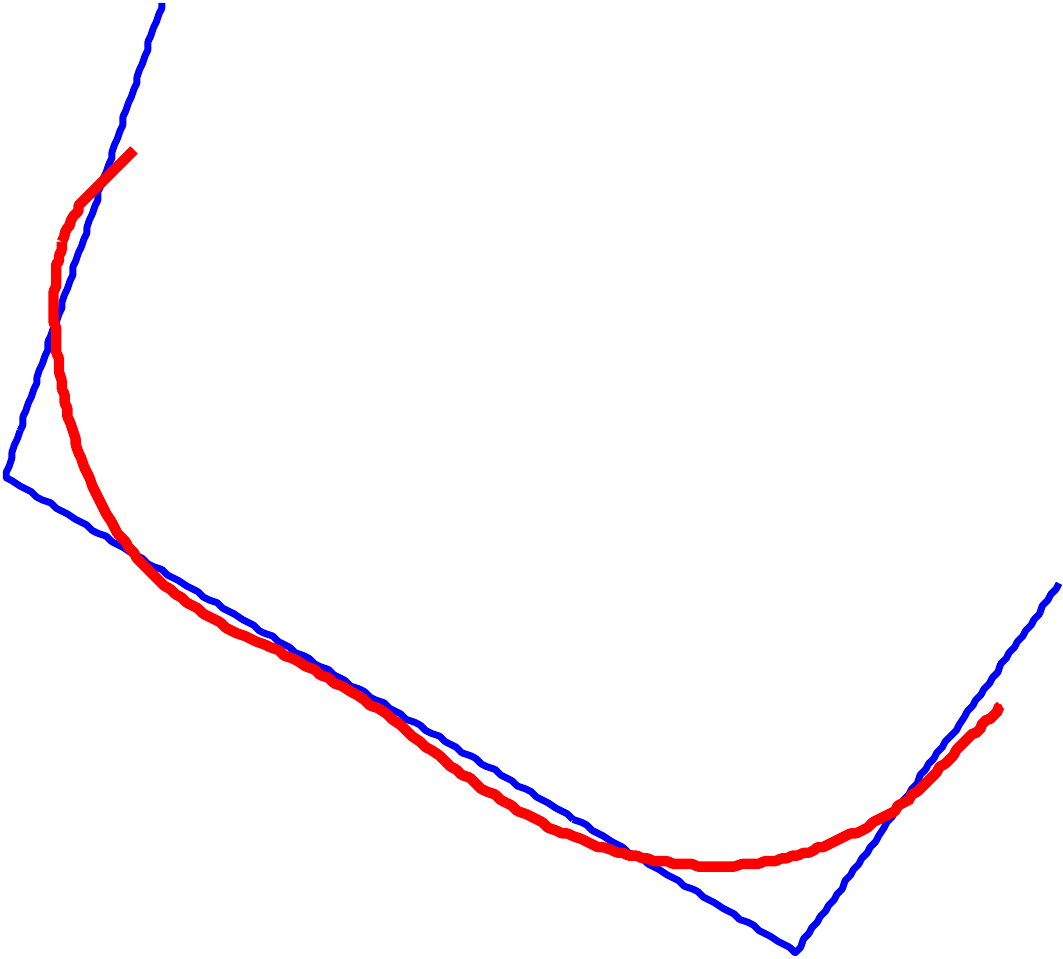}&
\includegraphics[width=.26\linewidth]{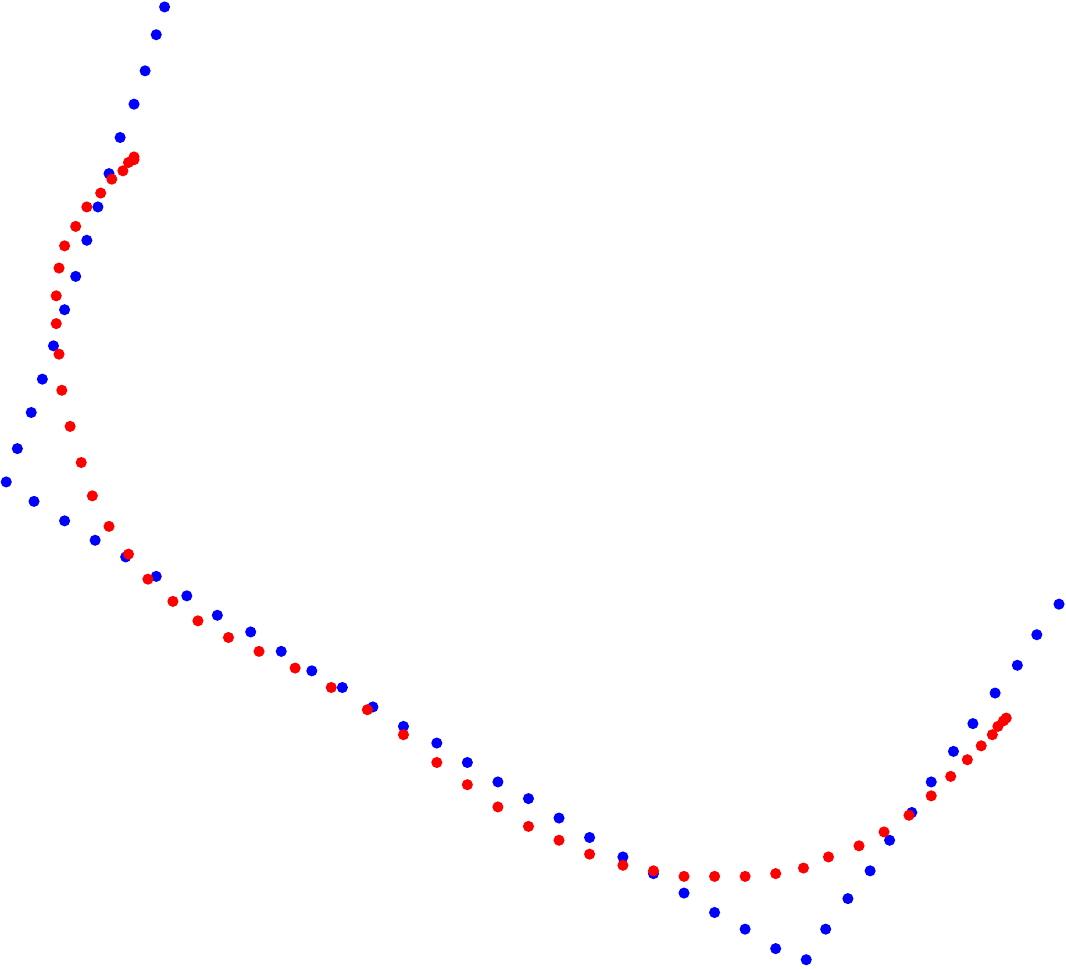}&
\hspace{-.03\linewidth}
\rotatebox{90}{\hspace{0.08\linewidth} $g(t)$} &
\hspace{-.052\linewidth}
\includegraphics[width=.33\linewidth]{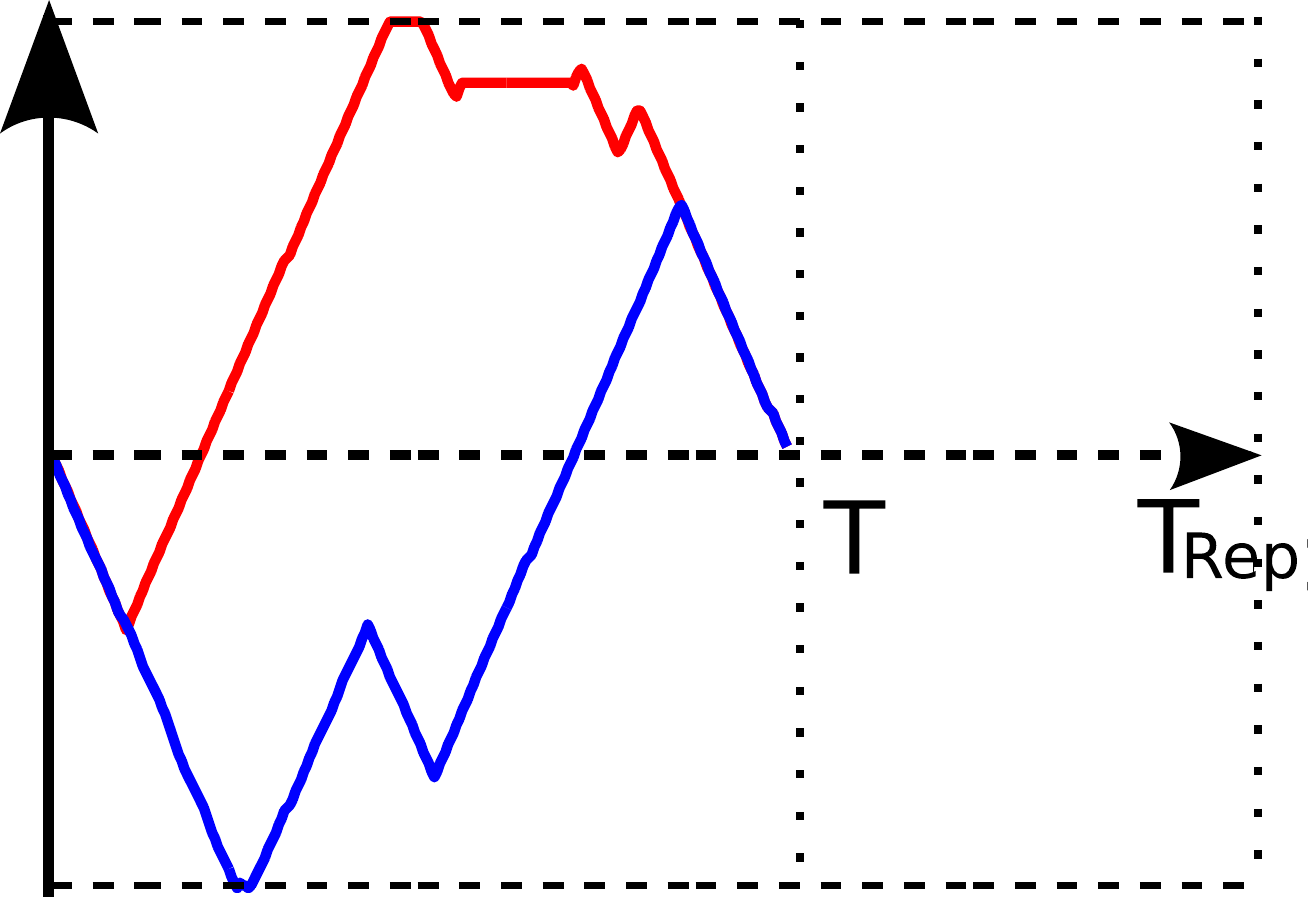}\\[-.01 \linewidth]
 & & & $t$ \\
\end{tabular}
\caption{\label{fig:illustr_proj} Comparison of two methods to design gradient waveforms.
\textbf{Top row:} Optimal control-based parameterization~\cite{Lustig08}. (a): input curve support. (b): discrete representation 
of the optimal reparameterization of the curve in $\mathcal{S}$. (c): corresponding gradient waveforms
($\color{red}{g_x}, \color{blue}{g_y}$). Dashed lines correspond to $0$ and +/- $G_{\max}$. \textbf{Bottom row:}
Illustration of the projection algorithm. (d): the same \textcolor{blue}{input} curve as in
(a) parameterized at maximal speed, and the support of the
\textcolor{red}{projected} curve onto the set $\mathcal{S}$. (e): discrete representation of the
\textcolor{blue}{input} and \textcolor{red}{projected} curves. (f): corresponding gradient waveforms
($\color{red}{g_x}, \color{blue}{g_y}$) with the same time scale as in (c): time 
to traverse the curve is reduced by \emph{39\%}.}
\end{figure}

In the next section, we will explain why the empirical distribution of the curves obtained by this projection method is close to the one of the input curve, and we illustrate how the parameterization will distort the sampling distribution.


\section{Control of the sampling density}
\label{sec:VDS}


Recent works have emphasized the importance of the sampling density~\cite{Chauffert14,Puy11,Krahmer12,Adcock13} in an attempt to reduce the amount of acquired data while preserving image quality at the reconstruction step. The choice of an accurate distribution $p$ is crucial since it directly impacts the number of measurements required~\cite{Rauhut10}.
In this paper, we will denote by $\pi$ a distribution defined over the $k$-space $K$. The profile of this distribution can be obtained by theoretical arguments~\cite{Chauffert14,Puy11,Krahmer12,Adcock13} leading to distributions as the one depicted in Fig.~\ref{fig:sampling_distribution}(a). Some heuristical distributions are known to perform well in CS-MRI experiments~(Fig.~\ref{fig:sampling_distribution}(b)). A comparison between these two approaches can be found in~\cite{Chauffert13}. 

\begin{figure}[!ht]
\begin{center}
\begin{tabular}{cc}
(a)&(b)\\
\includegraphics[width=.3\linewidth]{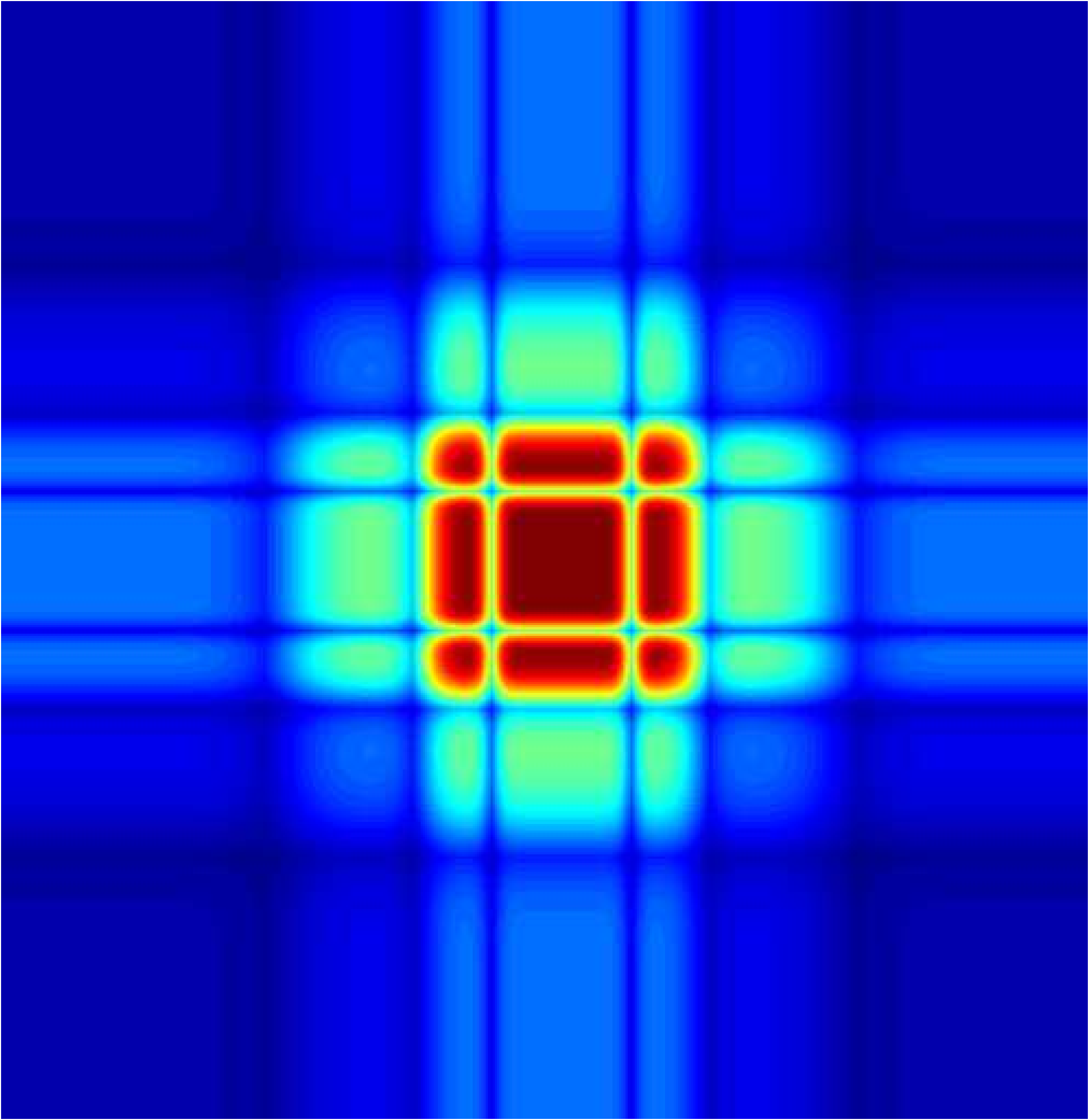}&
\includegraphics[width=.3\linewidth]{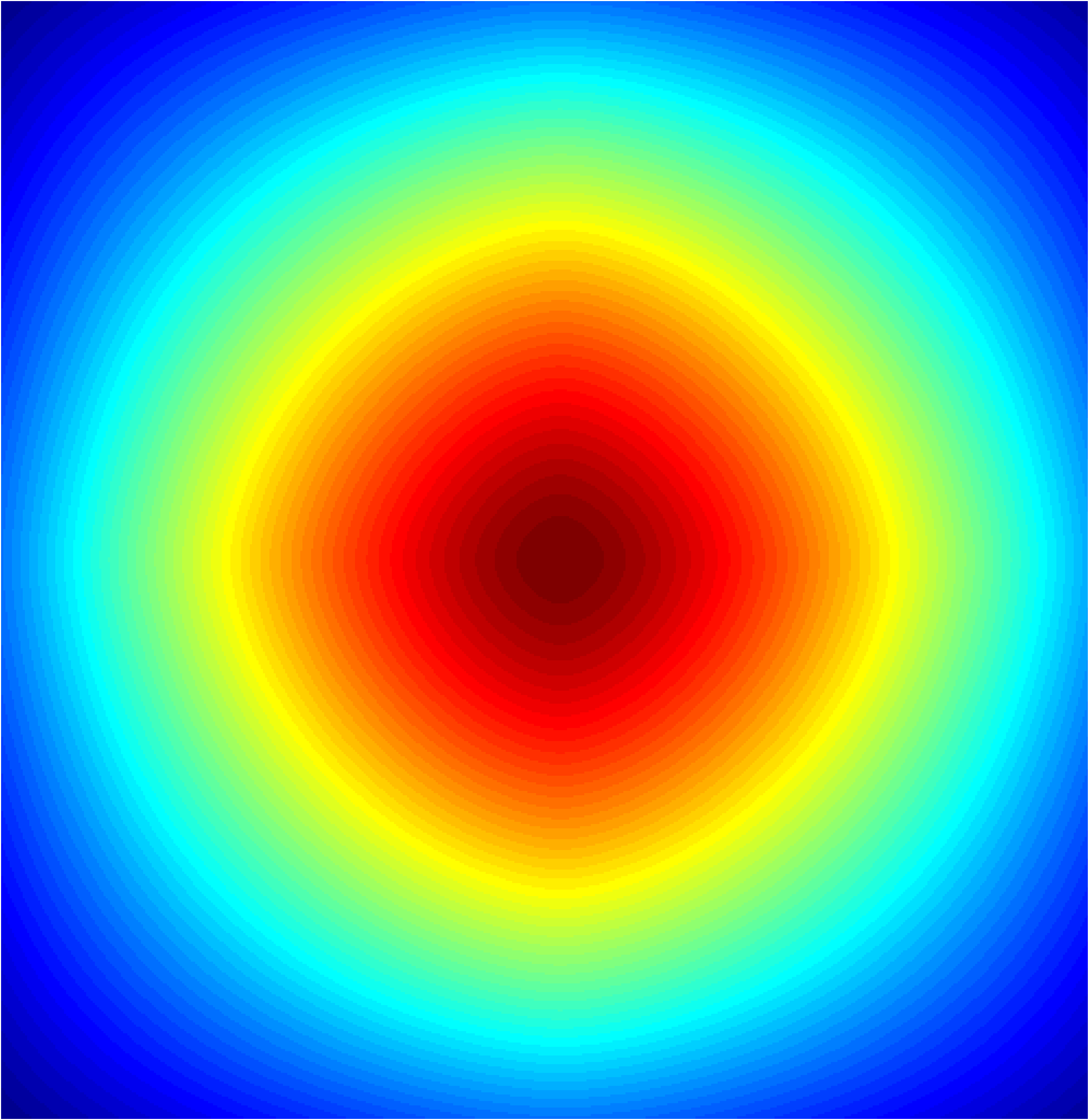}
\end{tabular}
\end{center}
\caption{\label{fig:sampling_distribution} Examples of 2D sampling distribution. (a): optimal distribution for a Symmlet transform~\cite{Puy11,Chauffert14}. (b): radial distribution advocated in~\cite{Lustig07}: $p(k)\propto (1-|k|/k_{\max})^p$ with $p=3$.}
\end{figure}

However, designing a trajectory that performs sampling according to a fixed distribution while satisfying gradient constraints, is really challenging and has never been adressed so far.
Recent attemps did not manage gradient constraints to design continuous variable density sampling trajectories since the sampling curve cannot be traversed at constant speed, especially at angular points as shown in Fig.~\ref{fig:illustr_proj}(b). This impacts the sampling density by concentrating probability mass in the portions of the curve associated with large curvature. 
In this section, we show that if we start from an inadmissible trajectory that covers the $k$-space with an empirical distribution close to the target one, then we can figure out how to design an admissible sampling trajectory with almost the same empirical distribution.
 Strategies to design such input curves are detailed in Appendix~\ref{sec:w2}. In what follows, we first derive theoretical guarantees, and then, we perform experiments to quantify the distribution distortion.

\subsection{Theoretical guarantees}
\label{sec:VDS1}
Let us show that the sampling distribution carried over a curve that our algorithm delivers can be close to a fixed target distribution. To this end, we start with the definition of the empirical distribution of a curve, and then we introduce a distance defined over distributions. 

\begin{definition}[Empirical measure of a curve]
Let $\gamma$ denote the Lebesgue measure and $\gamma_T=\frac{\gamma}{T}$ denote the Lebesgue measure normalized on the interval $[0,T]$. The \emph{empirical measure} of a curve $s:[0,T]\mapsto K \subseteq \R^d$ is defined for any measurable set $\omega$ of $K$ as:
\begin{align*}
P_s(\omega)=\gamma_T(s^{-1}(\omega)).
\end{align*}
\end{definition}

This means that the mass of a set $\omega$ is proportional to the time spent by the curve in $\omega$.
The natural distance arising in our work is the between-distribution Wasserstein distance $W_2$ defined hereafter:
\begin{definition}[Wasserstein distance $W_2$]
Let $M$ be a domain of $\R^d$ and $\mathcal{P}(M)$ be the set of measures over $M$. For $\mu,\nu \in \mathcal{P}(M)$, $W_2$ is defined as:
\begin{align}
W_2(\mu,\nu)= \left(\inf_{\gamma \in \Pi(\mu,\nu)} \int \|x-y\|_2^2 \dd \gamma(x,y) \right)^{\frac{1}{2}}
\end{align}
where $\Pi\subset \mathcal{P}(M\times M)$ denote the set of measures over $M\times M$ with marginals $\mu$ and $\nu$ on the first and second factors, respectively.
\end{definition}
$W_2$ is a distance over $\mathcal{P}(M)$~(see e.g., \cite{villani2008optimal}). Intuitively, if $\mu$ and $\nu$ are seen as mountains, the distance is the minimum cost of moving the mountains of $\mu$ into the mountains of $\nu$, where the cost is the $\ell_2$-distance of transportation multiplied by the mass moved. Hence, the coupling encodes the deformation map to turn one distribution into the other.

A global strategy to design feasible trajectories with an empirical distribution close to a target density $\pi$ relies on:
\begin{enumerate}
\item finding an input curve $c$, such that $P_c$ is close to $\pi$;
\item estimating $s^*$ the projection of $c$ onto the set of constraints, by solving Eq.~\eqref{pb:primal}.
\end{enumerate}

The objective is to show that $W_2(P_{s^*},\pi)$ is small. Since $W_2$ is a norm, the triangle inequality holds:
\begin{align}
\label{eq:triangle}
W_2(P_{s^*},\pi) \leqslant \underbrace{W_2(P_{s^*},P_c)}_{Distortion} + \underbrace{W_2(P_c,\pi)}_{Initial\ guess}.
\end{align}

The \emph{initial guess} term $W_2(P_c,\pi)$ can be as small as possible if $c$ is a Variable Density Sampler~\cite{Chauffert14}. The proof is postponed to Appendix~\ref{sec:w2} for the ease of reading. Here, we are interested in the \emph{distortion} term $W_2(P_{s^*},P_c)$. The following proposition shows that the $W_2$ distance between the empirical distributions of the input curve and of the output curve ($c$ and $s^*$) is controlled by the quantity $d(s^*,c)$ to be minimized when solving Eq.~\eqref{pb:primal}.
\begin{proposition}
\label{prop:W2} For any two curves $s$ and $c:[0,T] \to \R^d$:
\begin{align*}
W_2(P_c,P_s) \leqslant d(s,c).
\end{align*}
\end{proposition}

\begin{proof}
In terms of distributions, the quantity $d(s,c)$ reads:
\begin{align}
\label{pb:primal_measure}
d^2(s,c) = \int_{M\times M} \|x-y\|_2^2 \dd\gamma_{s,c}(x,y)
\end{align}
where $\gamma_{s,c}$ is the coupling between the empirical measures $P_s$ and $P_c$ defined for all couples of measure sets $(\omega_1,\omega_2)\in M^2$ by $\displaystyle \gamma_{c,s}(\omega_1,\omega_2)=\frac{1}{T} \int_{t=0}^T 1_{\omega_1}(s(t))1_{\omega_2}(c(t))\dd t$. The choice of this coupling is equivalent to choosing the transformation map as the association of locations of $c(t)$ and $s(t)$ for each $t$.
We notice that the quantity to be minimized in Eq.~\eqref{pb:primal_measure} is an upper bound of $W_2(P_s,P_c)^2$, with the specific coupling $\gamma_{s,c}$. \\
\end{proof}

To sum up, solving the projection problem amounts to minimizing an upper bound of the $W_2$ distance between the target and the empirical distributions of the solution if we neglect the influence of the {\it initial guess} term. 

\subsection{Simulations}

Next, we performed simulations to show that the sampling density is better preserved using our approach than relying on the optimal control approach. For doing so, we use travelling salesman-based~(TSP) sampling trajectories~\cite{Chauffert13c,Chauffert14}, which are an original way to design random trajectories which empirical distribution is any target distribution such as in~Fig.~\ref{fig:MCstudy1}(a). $10000$ such independent TSP were drawn and parameterized with arc-length: note that these parameterizations are not admissible in general. Then, we sampled each trajectory at fixed rate $\Delta t$ (as in~Fig.~\ref{fig:MCstudy1}(b)), to form an histogram (the empirical distribution shown in Fig.~\ref{fig:MCstudy1}(c)), which was compared to $\pi$ in Fig.~\ref{fig:MCstudy1}(d). The error is actually not closed to zero, since the convergence result is asymptotic in the size of the curve that is bounded in this experiment.

\begin{figure}[!ht]
\begin{center}
\begin{tabular}{cc}
(a)&(b)\\[-.01\linewidth]
\includegraphics[width=.3\linewidth]{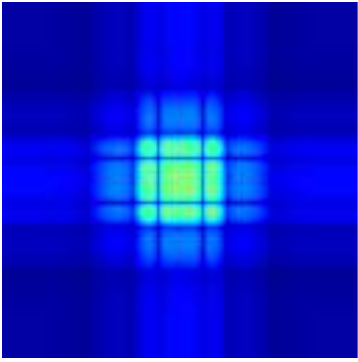}&
\includegraphics[width=.3\linewidth]{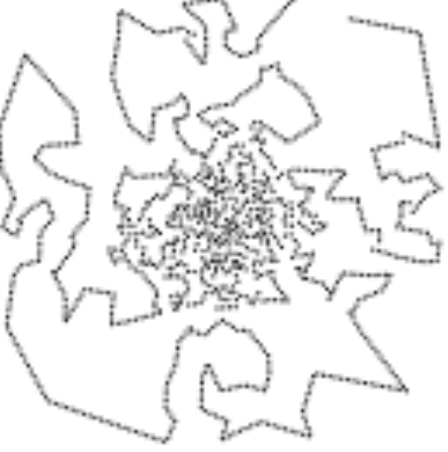}\\
(c)&(d)\\[-.01\linewidth]
\includegraphics[width=.3\linewidth]{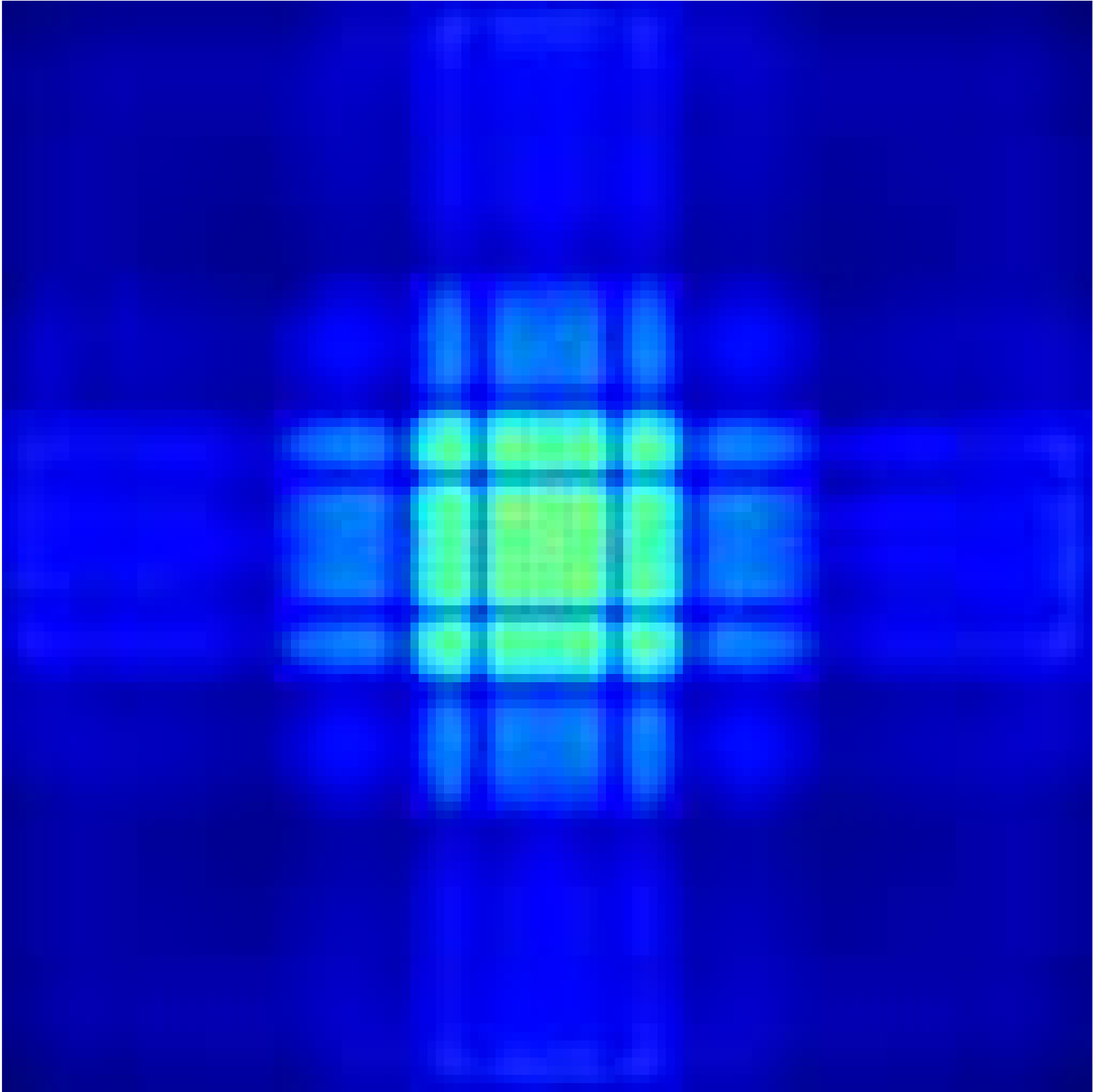}&
\includegraphics[width=.3\linewidth]{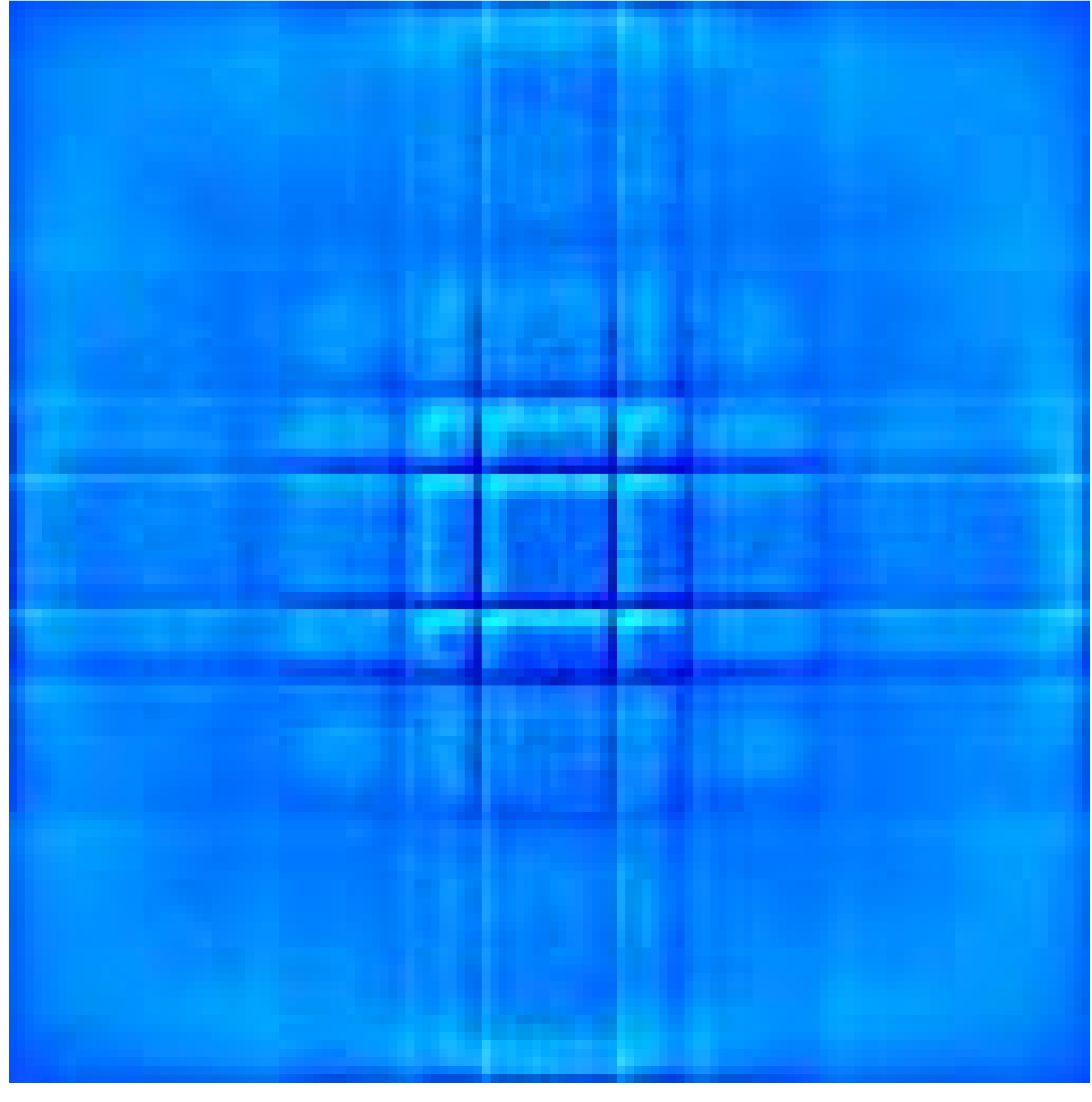}\\[-.01\linewidth]
& {\scriptsize rel. error = 12 \%}\\
\end{tabular}
\end{center}
\caption{\label{fig:MCstudy1} Illustration of TSP VDS. (a) Target distribution. (b) Regular sampling of TSP curve parameterized with its arc-length. (c) Empirical distribution~(histogram of samples). (d) Histogram difference (c)-(a).}
\end{figure}

In Fig.~\ref{fig:MCstudy2}~(top row), we show that the classical reparamete\-rization technique~\cite{Lustig08} lead to a major distortion of the sampling density, because of the behavior on the angular points illustrated in Fig.~\ref{fig:illustr_proj}(b). Then, we considered three constant speed parameterizations and projected them onto the same set of constraints ($G_{\max}=40$ mT.m$^{-1}$ and $S_{\max}=150$ mT.m$^{-1}$.ms$^{-1}$). Among these 3 initial guesses, we first used an initial parameterization with low velocity (10~\% of the maximal speed $\gamma G_{\max}$ with $\gamma=42.576$~MHz.T$^{-1}$), which projection fits the sampling density quite well. Then, we increased the velocity to reach first 50~\% and then 100~\% of the maximal speed. The distortion of the sampling density of the projected curve increased, but remained negligible in contrast to the exact reparameterization. Hence, this example illustrates that if we have a continuous trajectory with an empirical sampling distribution close to the target one, our projection algorithm enables to design feasible waveforms while sampling the $k$-space along a discretized trajectory with an empirical density close to the target one as well.

\begin{figure}
\begin{center}
\begin{tabular}{cccc}
 & {\footnotesize $k$-space trajectory} &  {\footnotesize emp. distribution} &  {\footnotesize diff. with $\pi$ }\\
\rotatebox{90}{\hspace{0.04\linewidth}\footnotesize reparameterization}&
\includegraphics[width=.25\linewidth]{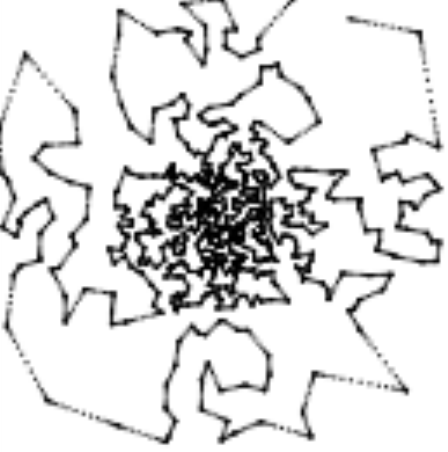}&
\includegraphics[width=.27\linewidth]{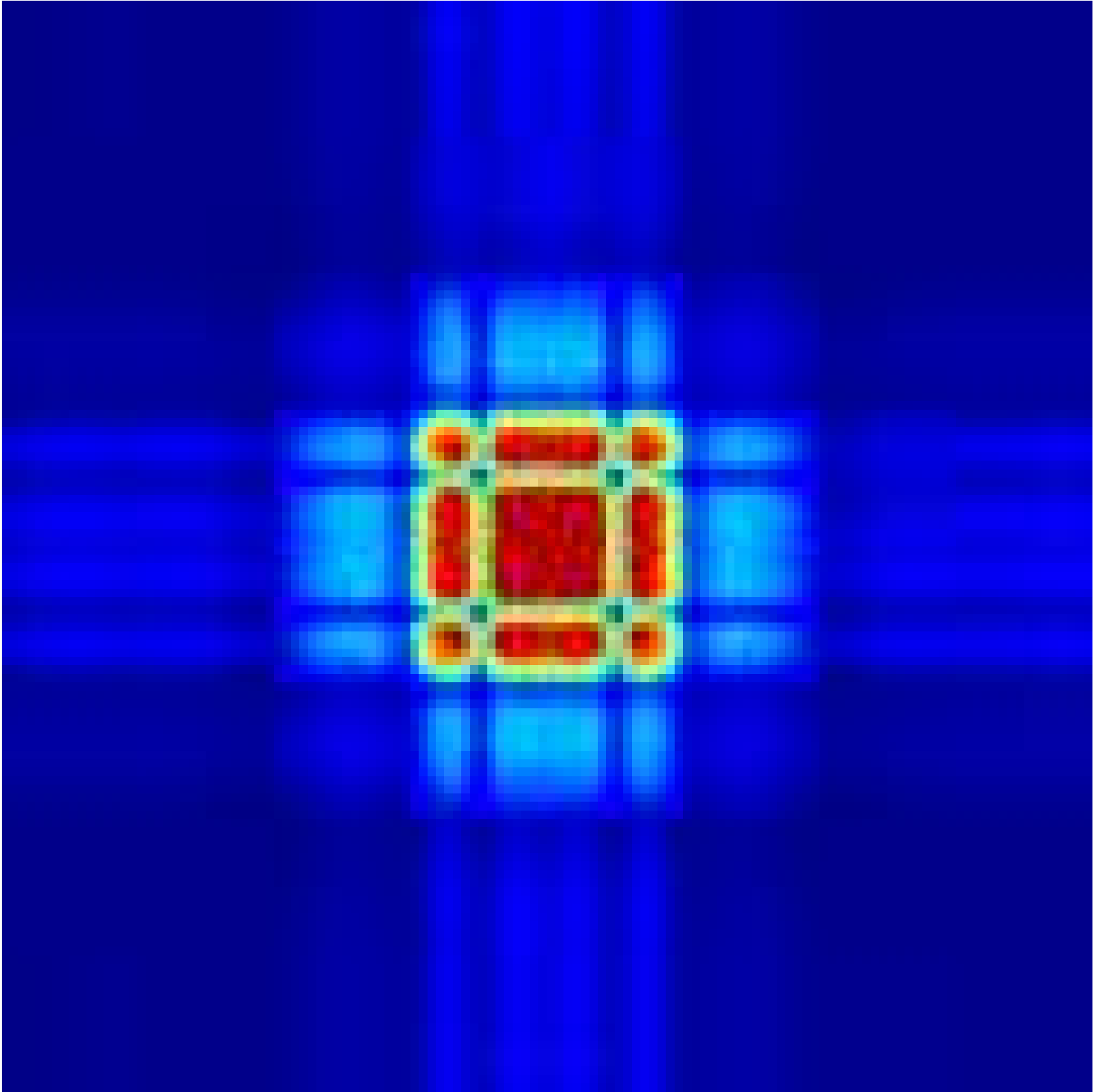}&
\includegraphics[width=.27\linewidth]{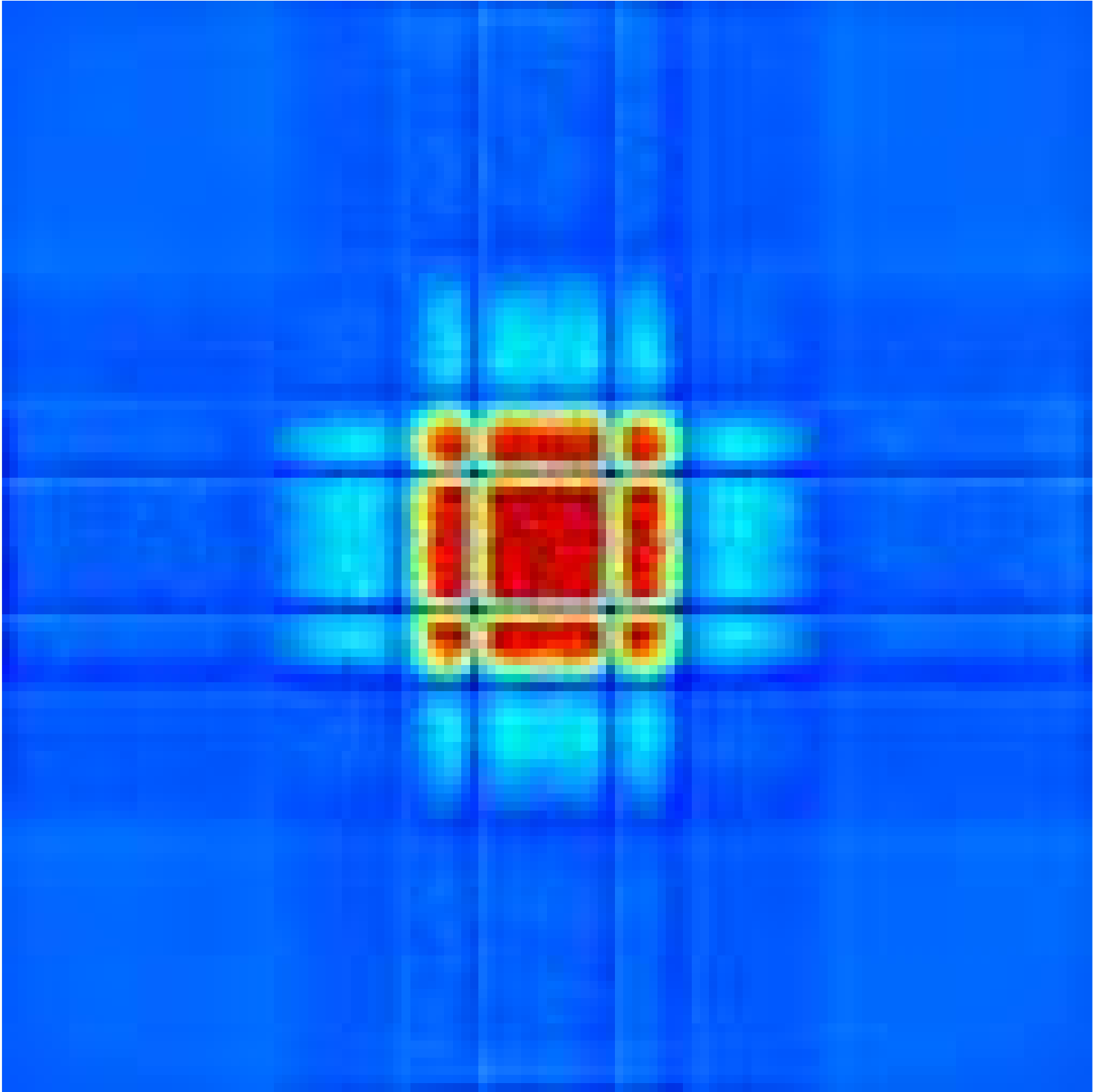}\\
 &{\scriptsize $T_{\texttt{Rep}}$ = 92 ms} & &  {\scriptsize rel. error = 60 \%}\\

\rotatebox{90}{\hspace{0.04\linewidth}\footnotesize 10 \% max. speed}&
\includegraphics[width=.25\linewidth]{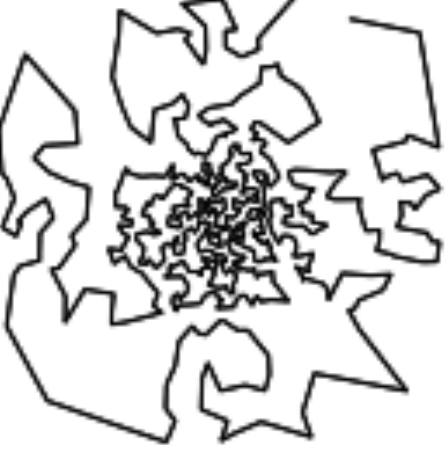}&
\includegraphics[width=.27\linewidth]{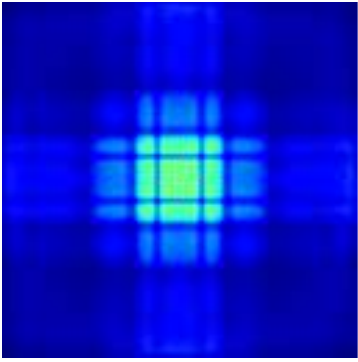}&
\includegraphics[width=.27\linewidth]{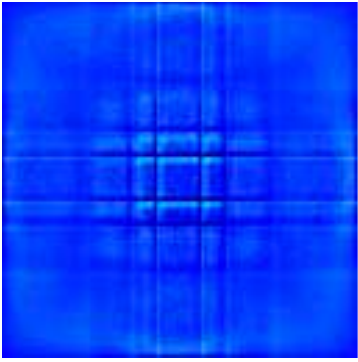}\\
 &{\scriptsize $T$ = 90 ms} & &  {\scriptsize rel. error = 10 \%}\\

\rotatebox{90}{\hspace{0.04\linewidth}\footnotesize 50 \% max. speed}&
\includegraphics[width=.25\linewidth]{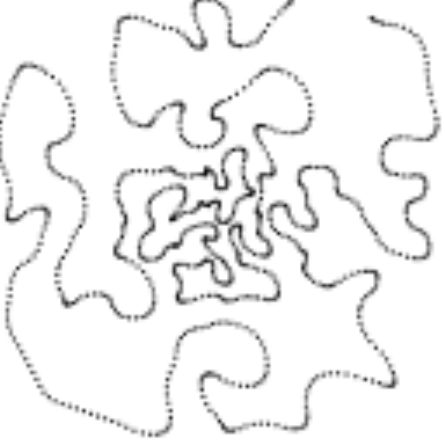}&
\includegraphics[width=.27\linewidth]{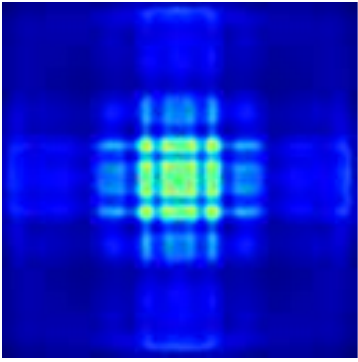}&
\includegraphics[width=.27\linewidth]{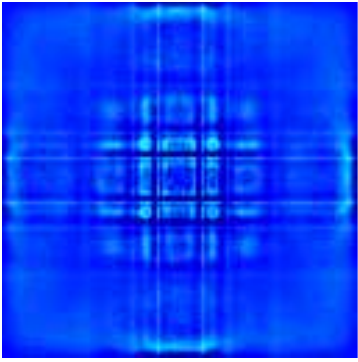}\\
 &{\scriptsize $T$ = 18 ms} & &  {\scriptsize rel. error = 12 \%}\\

\rotatebox{90}{\hspace{0.04\linewidth}\footnotesize max. speed}&
\includegraphics[width=.25\linewidth]{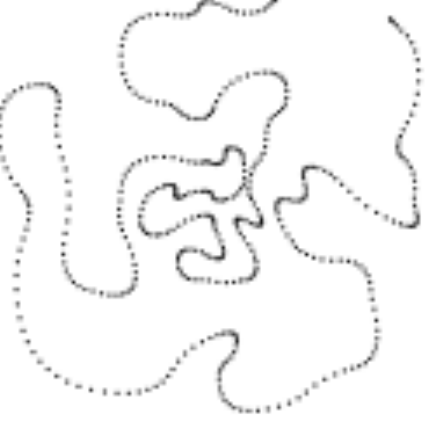}&
\includegraphics[width=.27\linewidth]{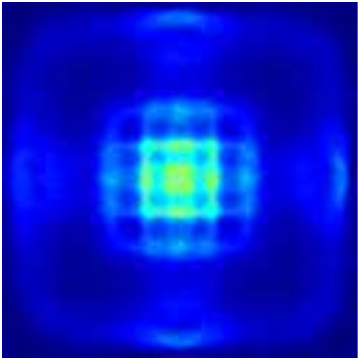}&
\includegraphics[width=.27\linewidth]{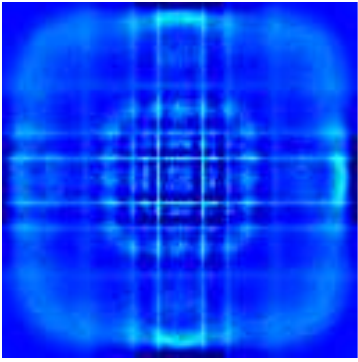}\\
 &{\scriptsize $T$ = 9 ms} & &  {\scriptsize rel. error = 14 \%}\\
\end{tabular}
\end{center}
\caption{\label{fig:MCstudy2} Illustration of TSP trajectories traversed with optimal control~(\textbf{top row}) and with our projection algorithm~(\textbf{rows 2-4}). Columns represent the $k$-space trajectory, the empirical distribution and the difference with the target distribution $\pi$~(Fig.~\ref{fig:MCstudy1}(a)).}
\end{figure}

\section{Finding feasible waveforms using convex optimization}
\label{sec:resol}
Since the set of constraints $\mathcal{S}\cap \mathcal{A}$ is convex, closed and non-empty, Problem \eqref{pb:primal} always admits a unique solution. 
Even though $\mathcal{S}$ has a rather simple structure\footnote{it is just a polytope when the $\ell^\infty$-norm is used.}, it is unlikely that an explicit solution to Problem \eqref{pb:primal} can be found. In what follows, we thus propose a numerical algorithm to find the projection.

\subsubsection*{Problem discretization}
In order to define a numerical algorithm, we first discretize the problem. 
A discrete curve $\bs$ 
is defined as a vector in $\R^{n\cdot d}$ where $n$ is the number of discretization points. Let $\bs(i) \in \R^d$ denote the curve location at time $(i-1) \delta t$ with $\delta t= \frac{T}{n-1}$. 
The discrete derivative $\dot \bs\in \R^{n\cdot d}$ is defined using first order differences:
\begin{align*}
 \dot \bs(i) = \left\{
    \begin{array}{ll}
        0 & \mbox{if } i=1, \\
        (\bs(i)-\bs(i-1))/\delta t \quad & \mbox{if }i \in \{2, \hdots ,n\}.
    \end{array}
\right.
\end{align*}
In the discrete setting, the first-order differential operator can be represented by a matrix $\D\in \R^{nd \times nd}$, i.e. $\dot \bs = \D \bs$. We define the discrete second-order differential operator by $\DD=-\D^* \D \in \R^{nd\times nd}$.

\subsubsection*{An efficient projection algorithm}
The discrete problem we consider is the same as problem \eqref{pb:primal} except that all objects are discretized. The set $\mathcal{S}$ is now $\bS \!:=\!\{\bs \in \R^{n\cdot d}, \|\D \bs\|\leqslant \alpha, \|\DD \bs\|\leqslant \beta\}$ with all norms discretized. Similarly, the discretized version of $\mathcal{A}$ is denoted $\bcA$.
The main idea of our algorithm is to take advantage of the structure of the dual problem to design an efficient projection algorithm. 
The following proposition specifies the dual problem and the primal-dual relationships.
\begin{proposition}
\label{prop:dual}
Let $\displaystyle \|\bq'\|_*:=\sup_{\|\bq\|\leq 1} \langle \bq,\bq'\rangle$ denote the dual norm of $\|\cdot\|$. 
The following equality holds
\begin{align}
& \min_{\bs\in \bS\cap \bcA} \frac{1}{2} \|\bs-\bc\|_2^2 \label{eq:PP}\\
& = \sup_{\bq_1,\bq_2\in \R^{n\cdot d}} F(\bq_1,\bq_2) -\alpha \| \bq_1\|_* -\beta \| \bq_2\|_* \label{eq:PD},
\end{align}
where 
\begin{equation}\label{eq:defF}
F(\bq_1,\bq_2) = \min_{\bs \in \bcA} \langle \D \bs, q_1\rangle + \langle \DD \bs, q_2\rangle + \frac{1}{2}\|\bs-\bc\|_2^2.
\end{equation}

Moreover, let $(\bq_1^*,\bq_2^*)$ denote any minimizer of the dual problem \eqref{eq:PD}, $\bs^*$ denote the unique solution of the primal problem \eqref{eq:PP} and $\bs^*(\bq_1^*,\bq_2^*)$ denote the solution of the minimization problem \eqref{eq:defF}. Then $\displaystyle \bs^*= \bs^*(\bq_1^*,\bq_2^*)$.
\end{proposition}
The following proposition gives an explicit expression of $\bs^*(\bq_1^*,\bq_2^*)$.
\begin{proposition}\label{prop:defsstar}
The minimizer
\begin{equation*}
 \bs^*(\bq_1^*,\bq_2^*) = \argmin_{\bs \in \bcA} \langle \D \bs, \bq_1\rangle + \langle \DD \bs, \bq_2\rangle + \frac{1}{2}\|\bs-\bc\|_2^2
\end{equation*}
is given by 
\begin{equation}\label{eq:defsstar}
\bs^*(\bq_1,\bq_2) = \bz  + \bA^+(\bv - \bA\bz),
\end{equation}
where $\bA^+=\bA^*(\bA \bA^*)^{-1}$ denotes the pseudo-inverse of $\bA$ and $\bz = \bc - \D^* \bq_1 - \DD^*\bq_2$.
\end{proposition}
Let us now analyze the smoothness properties of $F$.
\begin{proposition}\label{prop:regularityF}
Function $F(\bq_1,\bq_2)$ is concave differentiable with gradient given by
\begin{equation}\label{eq:defgrad}
\nabla F(\bq_1,\bq_2) = - \begin{pmatrix} \D \bs^*(\bq_1,\bq_2) \\ \DD \bs^*(\bq_1,\bq_2) \end{pmatrix}.
\end{equation}

Moreover, the gradient mapping $\nabla F$ is Lipschitz continuous with constant  $L=||| \D^*\D\! +\! \DD^*\DD |||$, where $|||\bM|||$ denotes the spectral norm of $\bM$.
\end{proposition}

The proofs are given in Appendix~\ref{sec:proofp1}.
The dual problem \eqref{eq:PD} has a favorable structure for optimization: 
it is the sum of a differentiable convex function $\tilde{F}(\bq_1,\bq_2)=-F(\bq_1,\bq_2)$ and of a simple convex function $G(\bq_1,\bq_2)\!=\alpha \| \bq_1\|_* \!+\!\beta \| \bq_2\|_*$. 
The sum $\tilde F\!+\!G$ can thus be minimized efficiently using accelerated proximal gradient descents~\cite{Nesterov83} (see Algorithm~\ref{algo:max} below).
\begin{algorithm}[!ht]
\KwIn{$\bc \in \R^{n\cdot d}$, $ \alpha, \beta >0$, $n_{it}$.}
\KwOut{$\tilde \bs \in \R^{n \cdot d}$ an approximation of the solution $\bs^*$.}
\textbf{Initialize} $\bq^{(0)}=(\bq_1^{(0)}, \bq_2^{(0)})$ with $\bq_i^{(0)}=0$ for $i=1,2$. Set $\by^{(0)}=\bq^{(0)}$. \\
Set $t=1/L$. \\
\For{$k=1\hdots n_{it}$}{
  $\bq^{(k)}=\prox_{t G}(\by^{(k-1)}-t\nabla \tilde F(\by^{(k-1)}))$\\
  $\by^{(k)}=\bq^{(k)}+\frac{k-1}{k+2}(\bq^{(k)}-\bq^{(k-1)})$\\
}
\Return{$ \displaystyle \tilde \bs=\bs^*\left(\bq_1^{(n_{it})},\bq_2^{(n_{it})}\right)$}\;
\caption{Projection algorithm in the dual space}
\label{algo:max}
\end{algorithm}

Moreover, by combining the convergence rate results of~\cite{Nesterov83,beck2009gradient} and some convex analysis (see Appendix \ref{sec:prthm}), we obtain the following convergence rate:
\begin{theorem}\label{thm:cvrate}
Algorithm~\ref{algo:max} ensures that the distance to the minimizer decreases as $ \mathcal{O}\left( \frac{1}{k^2}\right)$:
\begin{equation}
 \| \bs^{(k)}-\bs^* \|_2^2 \leq \frac{2L \|\bq^{(0)}- \bq^*\|_2^2}{k^2}.
\end{equation}
\end{theorem}

\section{Illustrations}
\label{sec:illustr}
To compare our results with~\cite{Lustig08}, we used the same gradient constraints. In particular, the maximal gradient norm $G_{\max}$ is set to 40~mT.m$^{-1}$, and the slew-rate  $S_{\max}$ to 150~mT.m$^{-1}$.ms$^{-1}$. $k$-space size is 6~cm$^{-1}$ and sampling rate is $\Delta t= 4$ $\mu$s. For the ease of trajectory representation, we limit ourselves to 2D sampling curves, but the algorithm encompasses the 3D setting.
The Matlab codes containing the projection algorithm as well as the scripts to reproduce the results depicted hereafter are available at http://chauffertn.free.fr/codes.html.
First, we illustrate the output of Algorithm~\ref{algo:max} for the case of rosette trajectories introduced in~\cite{Noll97}, showing a similar behaviour of our algorithm compared to optimal control. Then, we show an application to gradient waveform design to traverse TSP-based sampling trajectories. In this case, only our algorithm yields a fast $k$-space coverage.

\subsection{Smooth trajectories}
The first illustration in Fig.~\ref{fig:rosette} shows that when the sampling curve is smooth enough, our algorithm has a similar behaviour than optimal control. Here, we considered the case of Rosette trajectory defined in~\cite{Noll97,Lustig08} as $k(t)=k_{\max} \sin (\omega_1 t) [\cos(\omega_2 t) ; \sin(\omega_2 t)]$ (with $\omega_1=1.419$ and $\omega_2=0.8233$ rad.s$^{-1}$). In Fig.~\ref{fig:rosette}, we considered $c(t)$ the parameterization of $k(t)$ at constant speed (90\% of $\gamma G_{\max}$). The black curve represents the projection of $c$ onto the set of constraints for Rotation Variant~(RV) constraints~(Fig.~\ref{fig:rosette}(a)) and for Rotation Invariant~(RIV) constraints~(Fig.~\ref{fig:rosette}(b)). Then, we show the corresponding gradient waveforms to traverse the trajectory~(see Fig.~\ref{fig:rosette}(c,d) for optimal control solution and Fig.~\ref{fig:rosette}(e,f) for the output of our projection algorithm).

We observed that the time spent to traverse the curve is slighty shorter compared to the exact reparameterization, while the support of the curve is reduced. Of course, the initial parameterization of $c$ has a strong impact, and there is a trade-off between the time spent to traverse the curve and the distance to the original support. We also noticed that RV constraints are less restrictive, enabling a faster traversal time in the optimal control framework~(Fig.~\ref{fig:rosette}(c)) or a lower support distortion in the proposed setting~(Fig.~\ref{fig:rosette}(a)), as also noticed in~\cite{Vaziri13}. 

In this example, there is no real difference between the optimal control and the projection approaches. The following example, based on piece-wise linear sampling curves, shows in contrast that optimal control yields really longer traversal time.

\begin{figure}[!ht]
\begin{center}
\begin{tabular}{cc}
RV constraints & RIV constraints \\
(a)&(b) \\
\includegraphics[width=.4\linewidth]{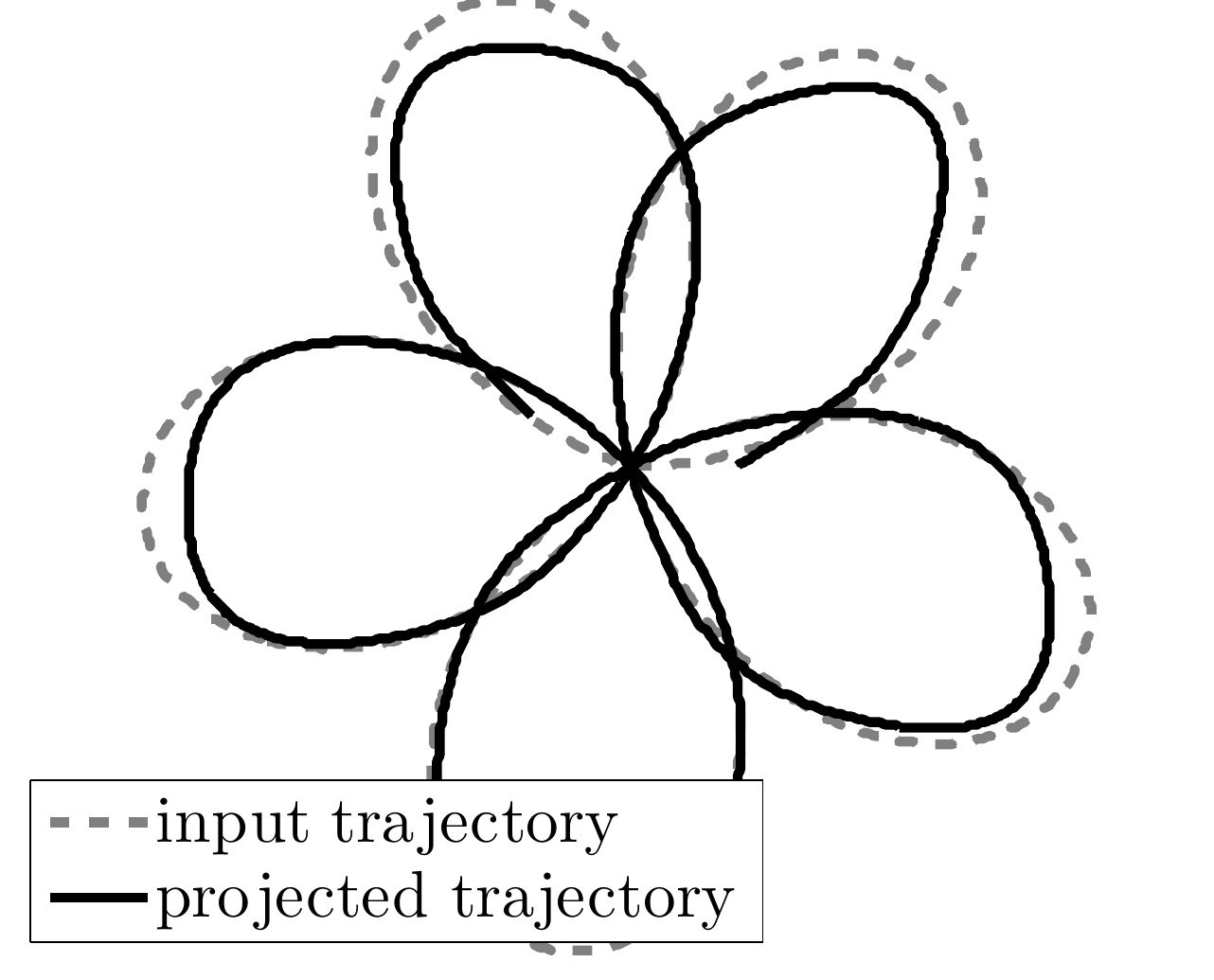}&
\includegraphics[width=.4\linewidth]{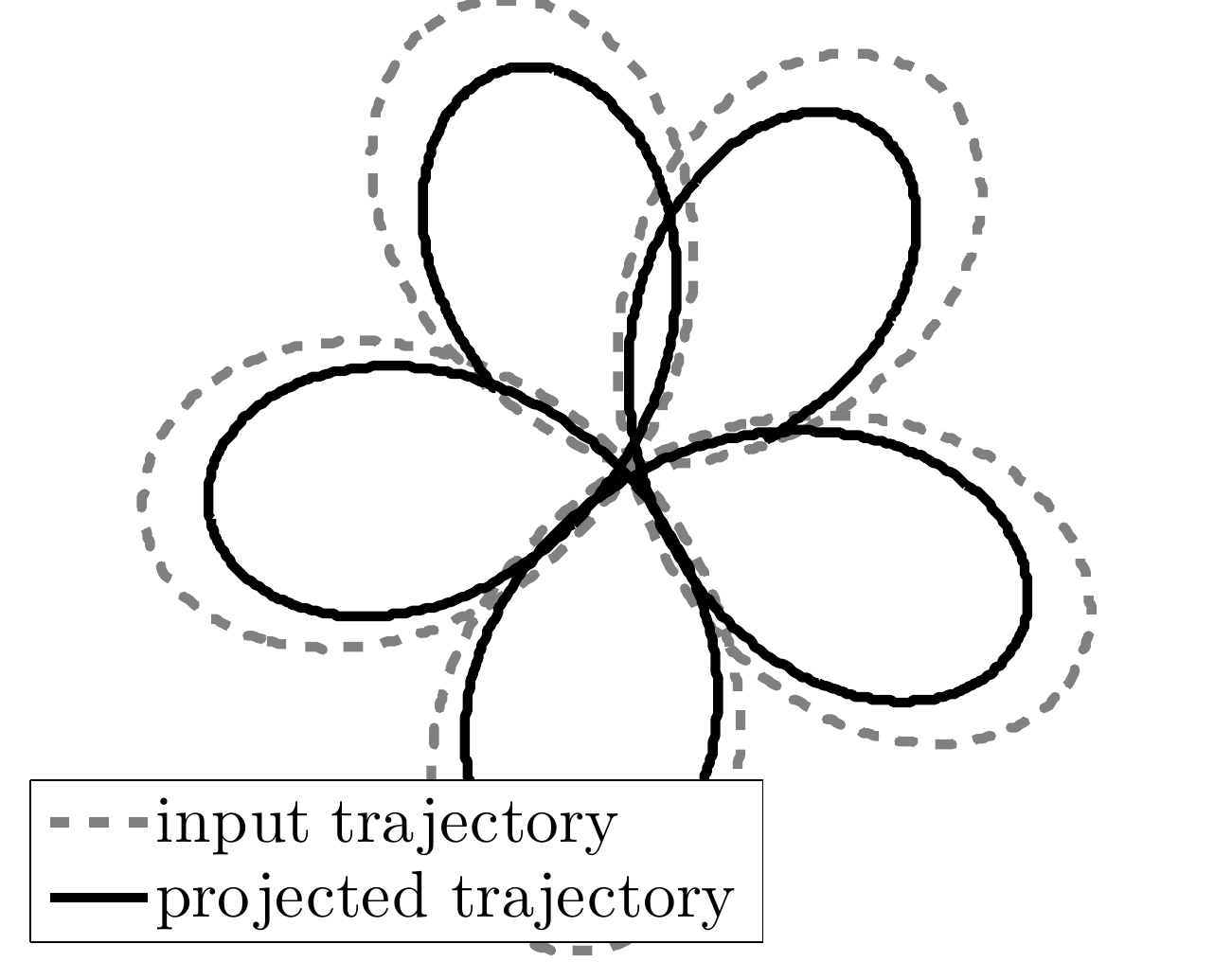}\\
\end{tabular}
\end{center}
\begin{center}

\begin{tabular}{cccc}
 &(c)&(d) \\[-.01 \linewidth]
\rotatebox{90}{\hspace{-0.01\linewidth} \footnotesize  reparameterization}
\hspace{.02\linewidth}
\rotatebox{90}{\hspace{0.08\linewidth} $g(t)$} &
\hspace{-.052\linewidth}
\includegraphics[width=.35\linewidth]{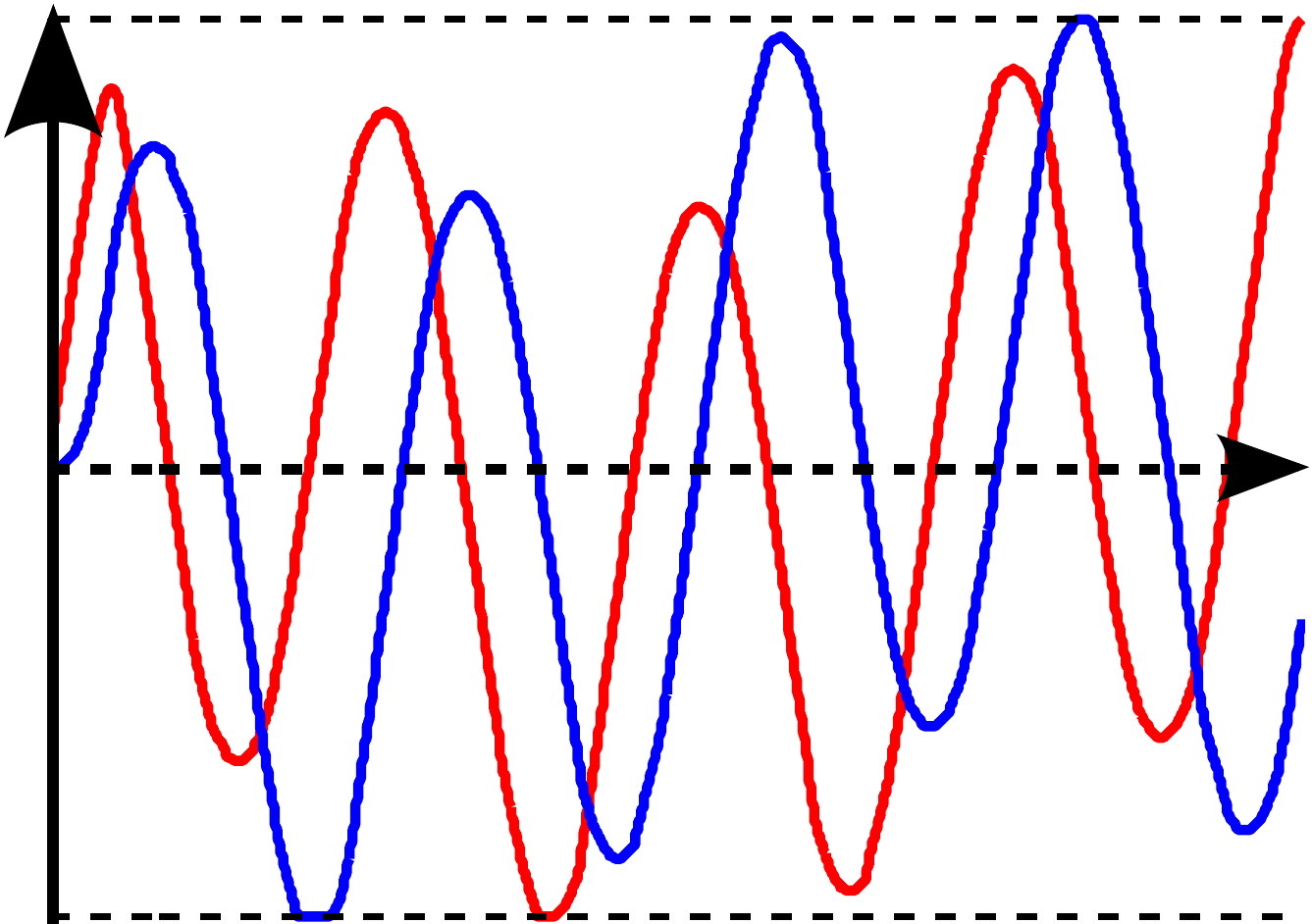}&
\includegraphics[width=.35\linewidth]{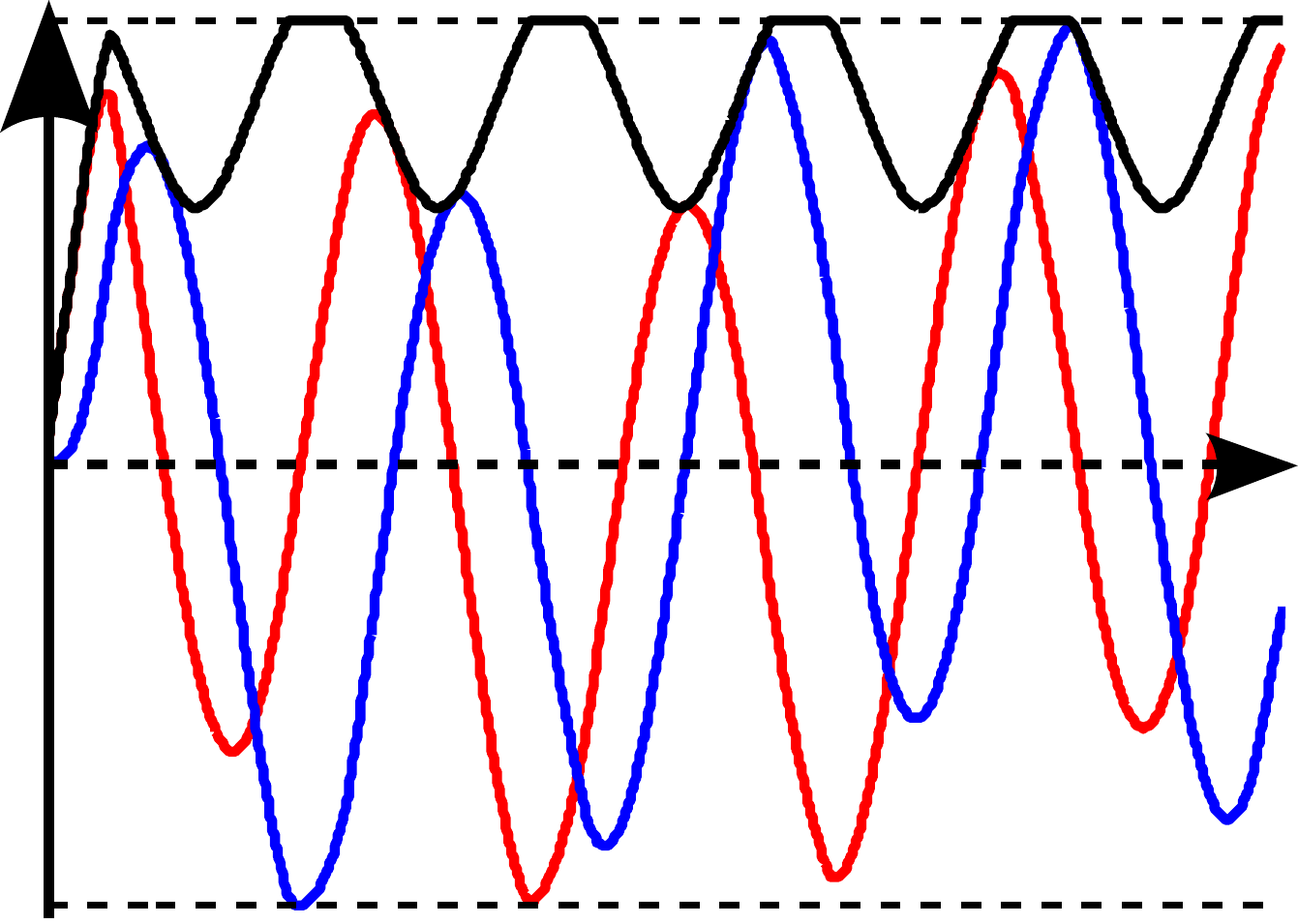}\\
 & {\scriptsize ($T_{\texttt{Rep}}$ = 5.45 ms)} & {\scriptsize ($T_{\texttt{Rep}}$ = 5.58 ms)} \\
 &(e)&(f) \\ [-.01 \linewidth]
\rotatebox{90}{\hspace{0.04\linewidth}\footnotesize   projection}
\hspace{.02\linewidth}
\rotatebox{90}{\hspace{0.08\linewidth} $g(t)$} &
\hspace{-.052\linewidth}
\includegraphics[width=.35\linewidth]{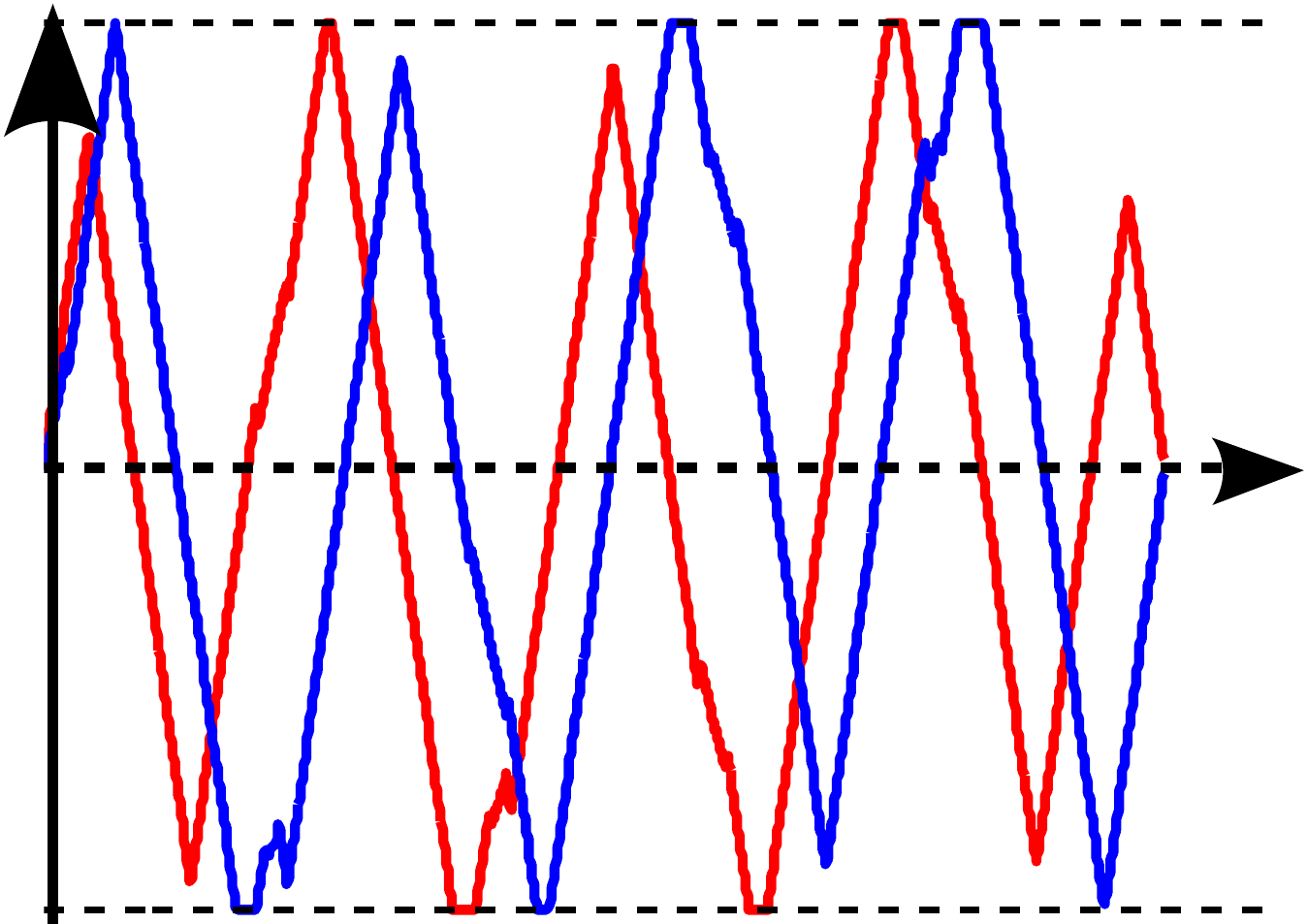}&
\includegraphics[width=.35\linewidth]{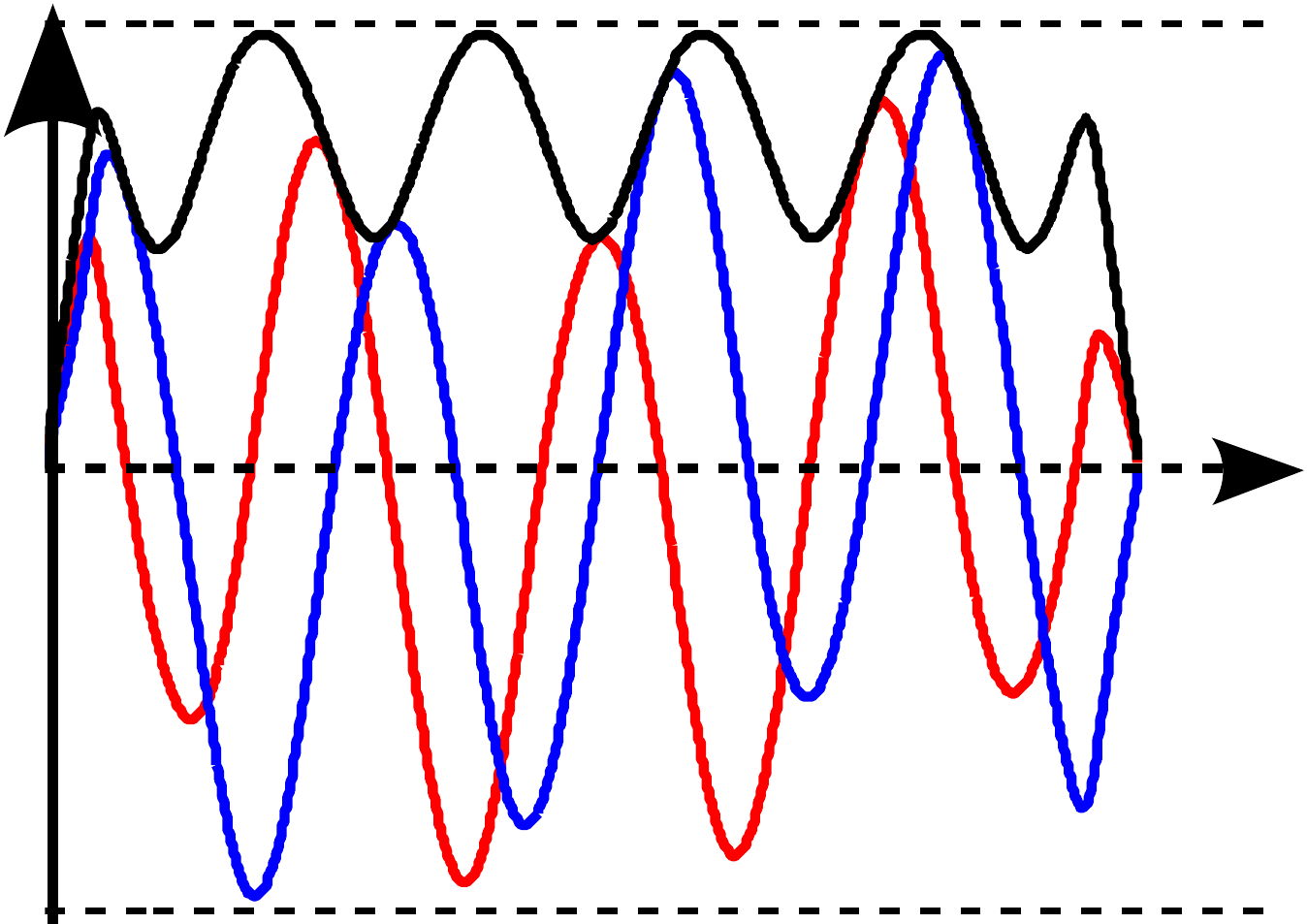}\\
 & {\scriptsize ($T$ = 4.94 ms)} & {\scriptsize ($T$ = 4.94 ms)} \\
\end{tabular}
\end{center}
\caption{\label{fig:rosette} Gradient waveforms to design rosette trajectories~\cite{Noll97}. Left (resp.~right), with Rotation-Variant~(RV) (resp.~Rotation-Invariant~(RIV)) constraints ($\|~\|=\|~\|_\infty$ (resp.~$=\|~\|_{\infty,2}$)). Corresponding gradient waveforms ($\color{red}{g_x(t)}$, $\color{blue}{g_y(t)}$, $\color{black}\|g(t)\|$). \textbf{Top row:} $k$-space trajectories between -6 and 6 cm$^{-1}$. \textbf{Middle row:} gradient waveforms corresponding to optimal reparameterization~\cite{Lustig08}. \textbf{Bottom row:} gradient waveforms corresponding to the proposed projection algorithm.}
\end{figure}

\subsection{TSP sampling}

Next, we performed similar experiments with a TSP trajectory, as introduced in~\cite{Chauffert13c}. We illustrate the results for RIV constraints, given an initial parameterization at 50~\% of the maximal speed. In Fig.~\ref{fig:TSPproj}(a), we notice that the output trajectory is a smoothed version of the initial piece-wise linear trajectory. Algthough the output curve has a slightly different support, the traversal time is reduced from 58~ms (optimal reparameterization that can be computed exactly) to 16~ms (Fig.~\ref{fig:TSPproj}(b-c)). 
The main reason is that TSP trajectories embody singular points, requiring the gradients to be set to zero for each of these points. Therefore, a sampling trajectory with singular points is time consuming. The main advantage of our algorithm is that the trajectory can be smoothed around these points, which saves a lot of acquisition time.

This second example shows that with existing methods, it is hopeless to implement TSP-based sequences in many MRI modalities, since the time to collect data can be larger than any realistic repetition time. In contrast, our method enables to traverse such curves in a reasonable time which can be tuned according to the image weighted~($T_1$, $T_2$, proton density).

\begin{figure}[!ht]
\begin{center}
\begin{tabular}{cl}
\begin{minipage}{.47\linewidth}
\begin{center}
(a) \\[.01 \linewidth]
\includegraphics[width=1.0\linewidth]{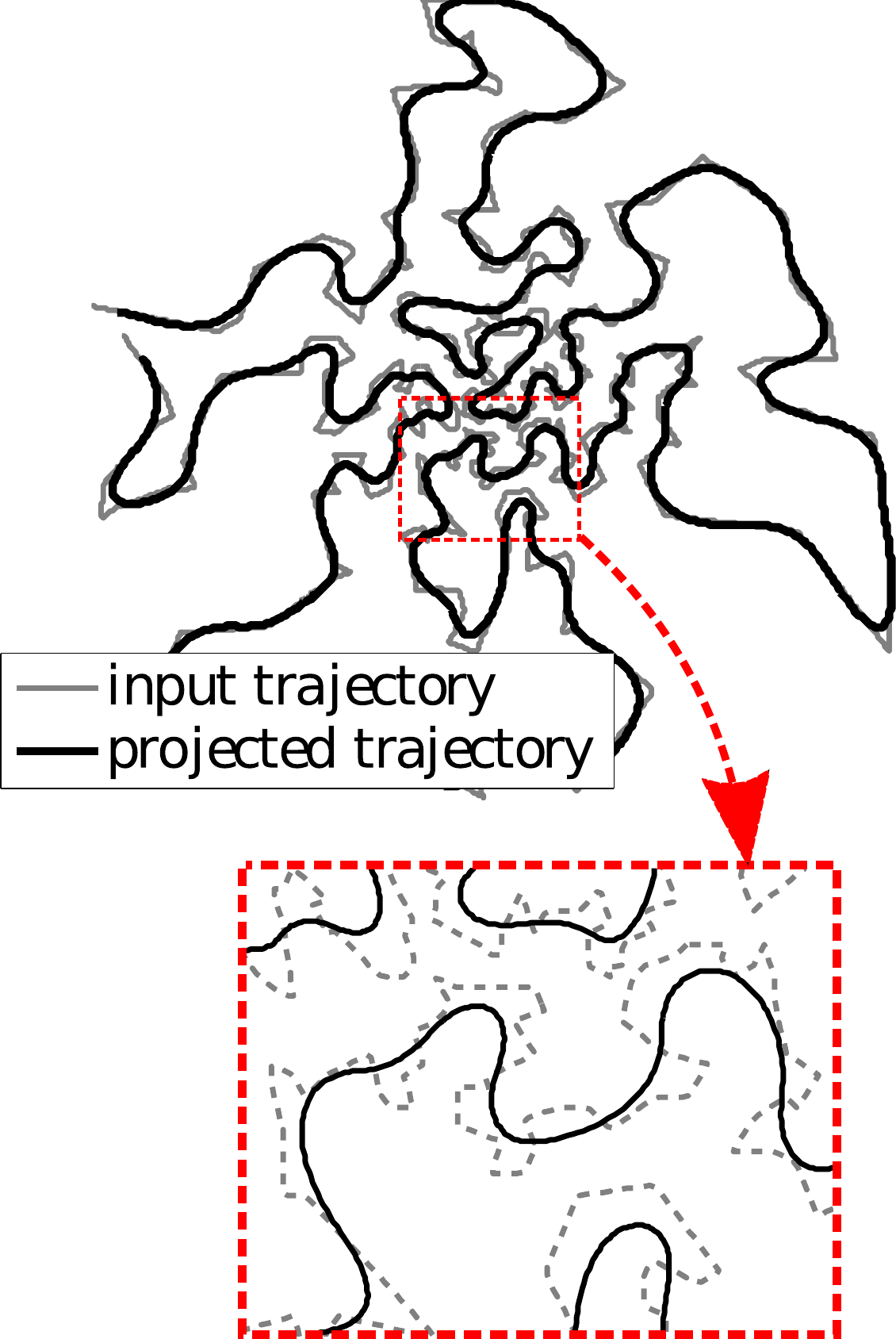}
\end{center}
\end{minipage}
&
\hspace{-.1\linewidth}
\begin{tabular}{c}
(b)\\[-.02 \linewidth]
\rotatebox{90}{\hspace{0.11\linewidth} $g(t)$} 
\hspace{-.03\linewidth}
\includegraphics[width=.5\linewidth]{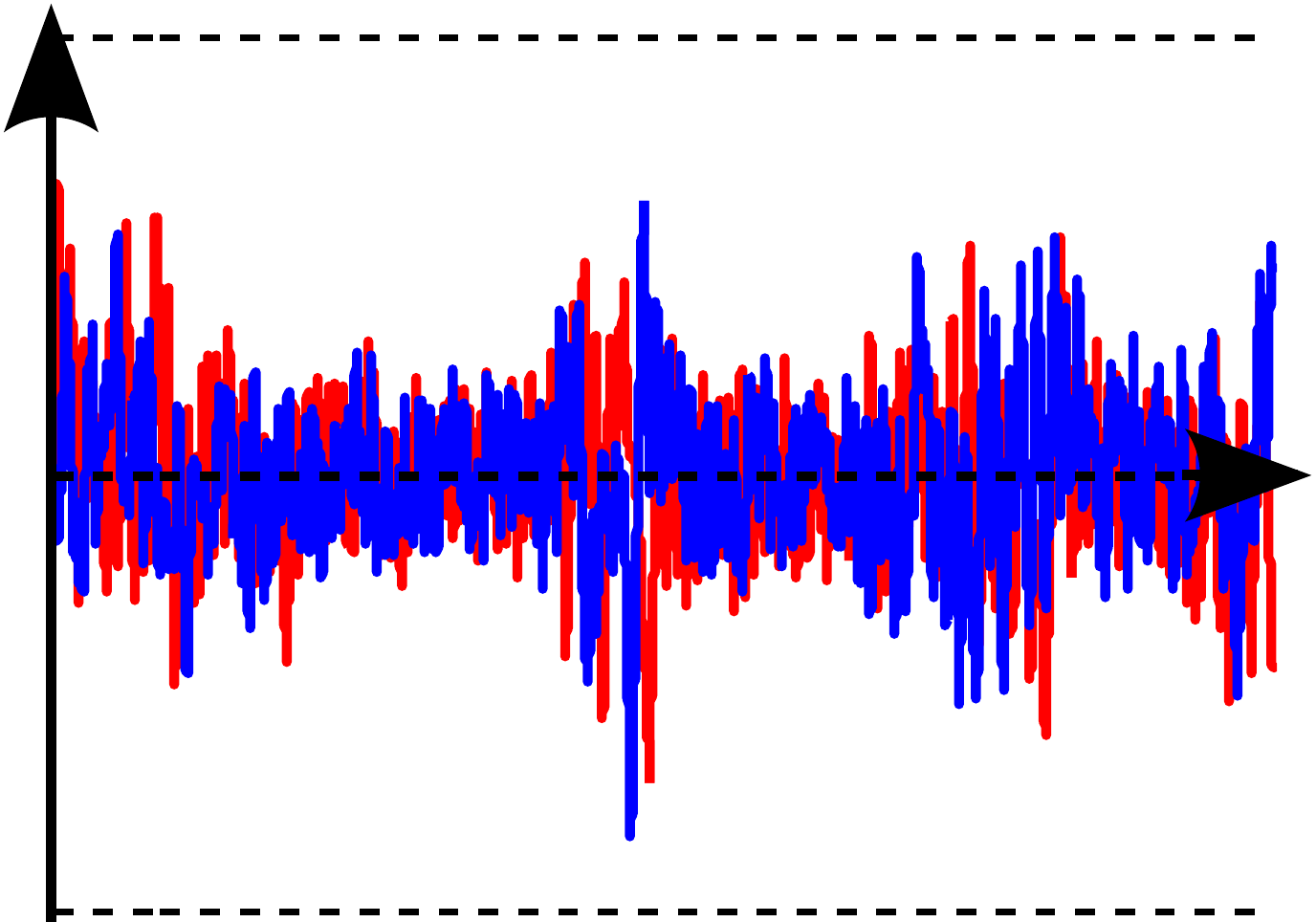}\\[-.02 \linewidth]
{\scriptsize ($T_{\texttt{Rep}}$ = 58 ms)} \\
(c)\\[-.02 \linewidth]
\rotatebox{90}{\hspace{0.12\linewidth} $g(t)$} 
\hspace{-.03\linewidth}
\includegraphics[width=.5\linewidth]{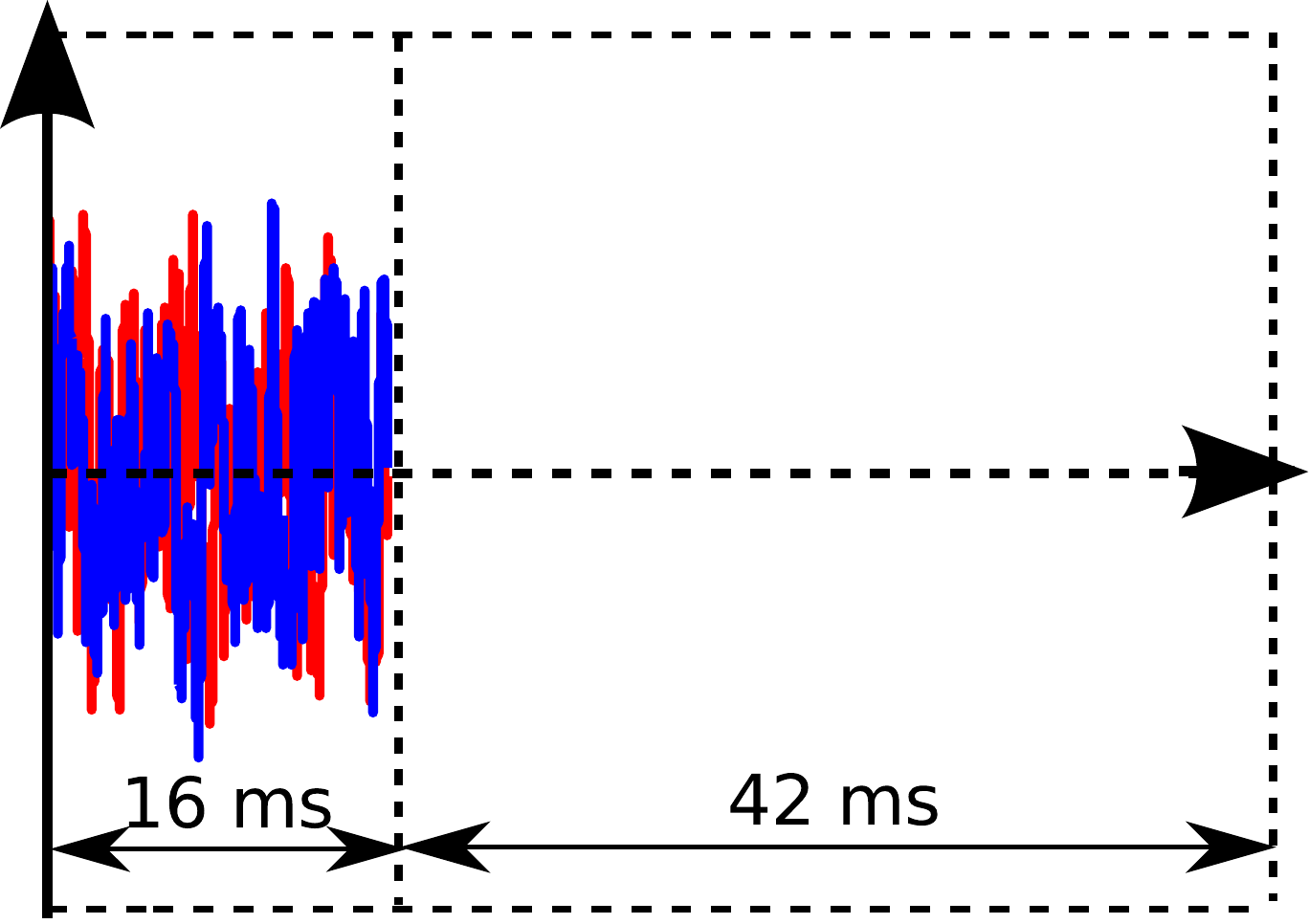}\\[-.02 \linewidth]
{\scriptsize ($T$ = 16 ms)} \\
\end{tabular}
\end{tabular}
\end{center}
\caption{\label{fig:TSPproj} Gradient waveforms to design TSP trajectories~\cite{Noll97}. (a): $k$-space trajectories between -6 and 6 cm$^{-1}$. (b): gradient waveforms ($g=(\color{red}{g_x},\color{blue}{g_y}$)) corresponding to optimal reparameterization~\cite{Lustig08}. (c): gradient waveforms corresponding to the proposed projection algorithm.}
\end{figure}

\section{Conclusion}

We developed an algorithm to project any parameterized curve onto the set of curves which can be implemented on actual MRI scanners. Our method is an alternative to the existing gradient waveform design based on optimal control. The main advantage of our approach is that one can reduce the time to traverse a curve if the latter has a lot of high curvature areas, like a solution of the TSP.

\section*{Acknowledgments}
This work benefited from the support of the FMJH Program Gaspard Monge in optimization and operation research, and from the support to this program from EDF.

\appendices

\section{Density deviation, control of $W_2$-distance.}
\label{sec:w2}

In Section~\ref{sec:VDS1}, we aim at controlling the Wasserstein distance $W_2(P_{s^*},\pi)$, where $\pi$ is a target fixed sampling distribution, and $P_{s^*}$ is the empirical distribution of the projected curve. We used the triangle inequlity~\eqref{eq:triangle} to bound this quantity by $W_2(P_{s^*},P_c)+ W_2(P_c,\pi)$. Here, we show that the quantity $W_2(P_c,\pi)$ can be as small as possible if $c$ is Variable Density Sampler~(VDS)~\cite{Chauffert14}. First, we define the concept of VDS, and then we provide two examples. Next, we show that if $c$ is a VDS,  $W_2(P_c,\pi)$ tends to 0 as the length of $c$ tends to infinity.

\subsection{Definition of a VDS}
First, we need to introduce the definition of weak convergence for measure:
\begin{definition}
A sequence of measures $\mu_n \in \mathcal{P}(K)$, the set of distributions defined over $K$, is said to weakly converge to $\mu$ if for any bounded continuous function $\phi$
\begin{align*}
\int_K \phi(x) \dd\mu_n(x) \to \int_K \phi(x) \dd\mu(x).
\end{align*}
We use the notation $\mu_n \rightharpoonup \mu$.
\end{definition}
According to~\cite{Chauffert14}, a (generalized) $\pi$-VDS is a set of times $T_n$, such that $T_n\to \infty$ when $n\to \infty$, and a sequence of curves $c_{T_n}:[0,T_n] \to \R^d$ such that $P_{c_{T_n}} \rightharpoonup \pi$ when $n$ tends to infinity. A consequence of the definition is that the relative time spent by the curve in a part of the $k$-space is proportional to its density. Before showing that this implies that $W_2(P_{c_{T_n}},\pi)$ tends to 0, we give two examples of VDS. 

\subsection{VDS examples}
\label{sec:examples}
We give two examples to design continuous sampling trajectories that match a given distribution. The two examples we propose provide a sequence of curves, hence a sequence of empirical measures, that weakly converge to the target density.

\subsubsection{Spiral sampling} The spiral-based variable density sampling is now classical in MRI~\cite{spielman1995magnetic}. For example, let $n\in\N$ be the number of revolutions and $r:[0,1]\mapsto\R^+$ a strictly increasing smooth function. Denote by $r^{-1}$ its inverse function. Define the spiral for $t\in [0,n]$ by $\displaystyle c_n(t)=r\Bigl(\frac{t}{n}\Bigr)\begin{pmatrix}\cos(2\pi t)\\ \sin(2\pi t) \end{pmatrix}$ and the target distribution $\pi$ by:
\begin{align*}
\pi(x,y)\!=\!\left\{\!\begin{array}{cl}
\frac{\overset{\text{\Large.}}{r^{-1}}\left( \sqrt{x^2+y^2}\right)}{2 \pi \int_{\rho=r(0)}^{r(1)} \overset{\text{\Large.}}{r^{-1}}(\rho)\rho \dd \rho} & \mbox{if} \quad r(0)\!\leqslant\! \sqrt{x^2\!+\!y^2}\!\leqslant \!r(1) \\
0 & \mbox{otherwise} \\
\end{array}
\right.
\end{align*}
then $P_{c_n}\rightharpoonup \pi$ when the number of revolutions $n$ tends to infinity.
\subsubsection{Travelling Salesman-based sampling} The idea of using the shortest path amongst a set of points (the ``cities") to design continuous trajectories with variable densities has been justified in~\cite{Chauffert13b,Chauffert14}. Let us draw $n$ $k$-space locations uniformly according to a density $q$ define over the $d$D $k$-space ($d=2$ or $3$), and join them by the shortest path (the Travelling Salesman solution). Then, denote by $c_n$ a constant-speed parameterization of this curve. Define the density:
\begin{align*}
\pi=\frac{q^{(d-1)/d}}{\int q^{(d-1)/d}(x)\dd(x)}
\end{align*}
Then $P_{c_n} \rightharpoonup \pi$ when the number of cities $n$ tends to infinity.

These two sampling strategies are efficient to cover the $k$-space according to target distributions, as depicted in Fig.~\ref{fig:exVDS}. For spiral sampling, the target distribution may be any 2D radial distribution, whereas the Travelling salesman-based sampling enable us to consider any 2D or 3D density.

\begin{figure}
\begin{center}
\begin{tabular}{cc}
(a)&(b)\\[-.01\linewidth]
\includegraphics[width=.3\linewidth]{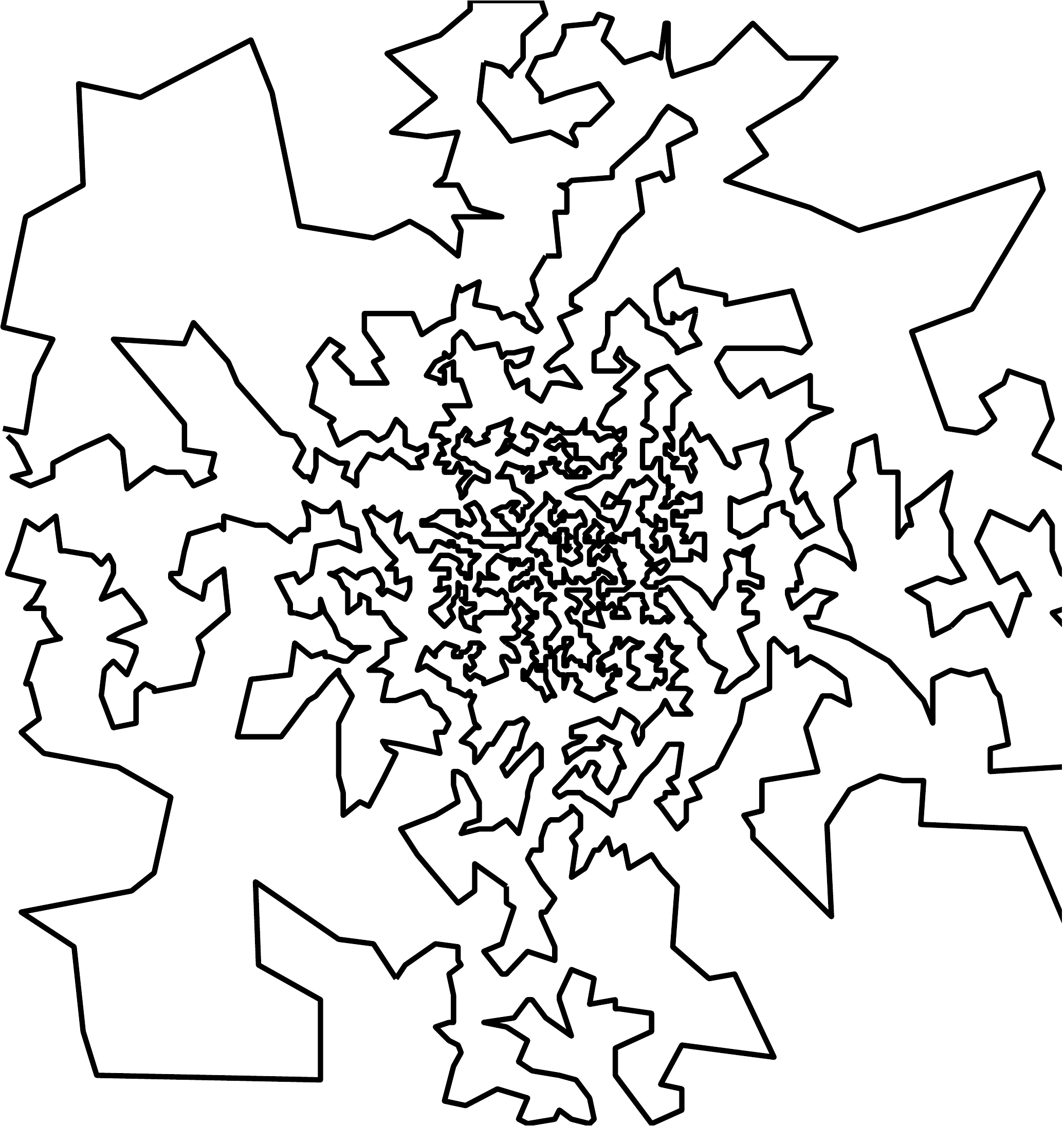}&
\includegraphics[width=.3\linewidth]{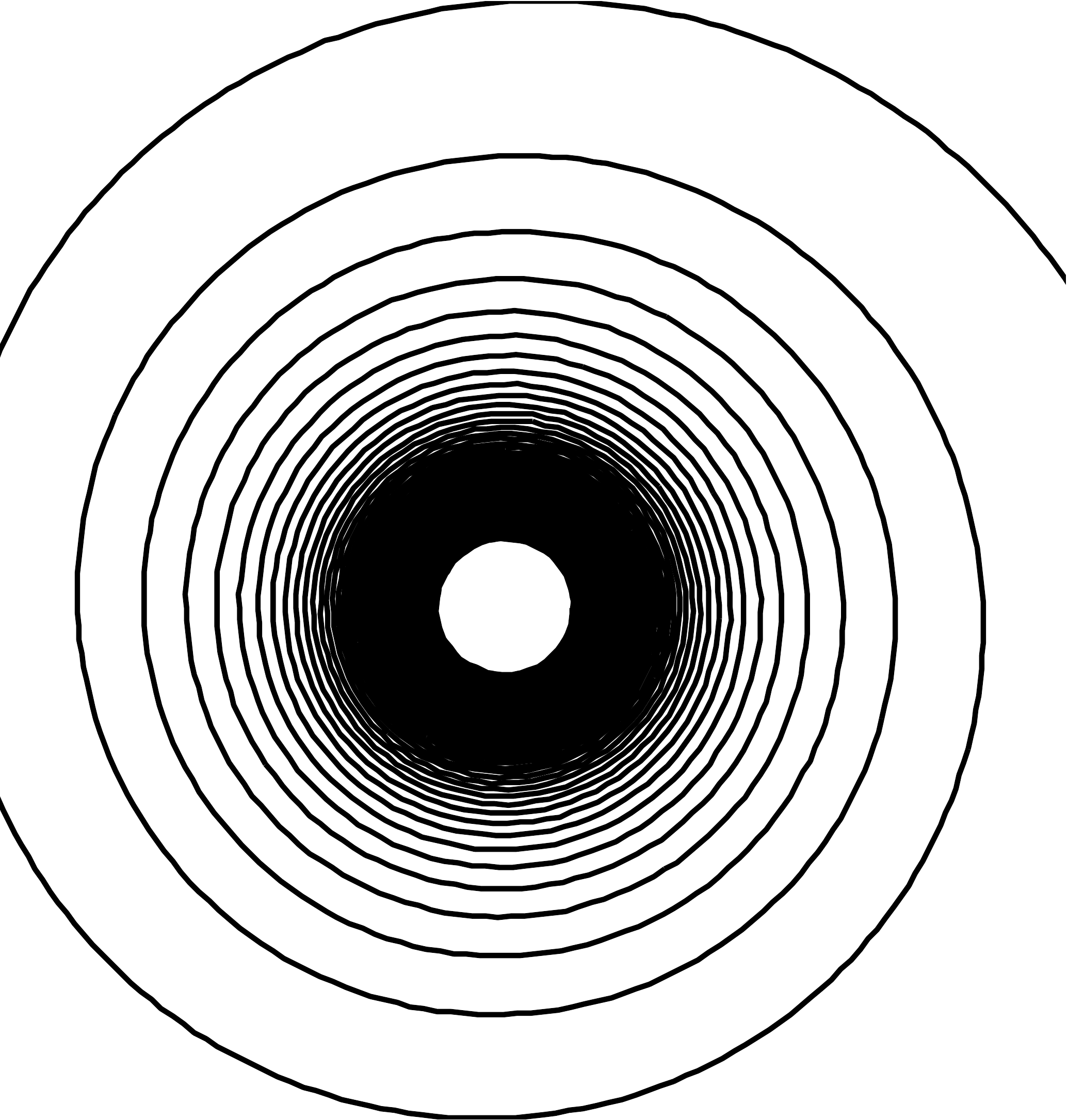}
\\
\end{tabular}
\end{center}
\caption{\label{fig:exVDS} Examples of Variable Density Sampling. (a): TSP trajectory with empirical distribution depicted in~Fig.~\ref{fig:sampling_distribution}(a).(b): spiral trajectory with empirical radial distribution such as~Fig.~\ref{fig:sampling_distribution}(b).}
\end{figure}

\subsection{Control of $W_2$ distance}
Let us now assume without loss of generality that $K=[-k_{\max}, k_{\max}]^d$. 

Let us recall a central result about $W_2$~(see e.g.,\cite{villani2008optimal}):
\begin{proposition}
Let $M\subset \R^d$, $\mu \in \mathcal{P}(M)$ and $\mu_n$ be a sequence of $\mathcal{P}(M)$.
Then, if $M$ is compact
\begin{align*} 
\mu_n \rightharpoonup \mu \Leftrightarrow W_2(\mu_n,\mu) \to 0
\end{align*} 
\end{proposition}

An immediate consequence of this proposition and of the compactness of $K$ is the following proposition:
\begin{proposition}
\label{prop:5}
Let $(c_{T_n})_{n\geqslant 1}$ be a $\pi$-VDS, and $\varepsilon>0$. Then, there exists $n\geqslant 1$ such that $c_{T_n}:[0,T_n]\to K$ fulfills:
\begin{align*}
W_2(P_{c_{T_n}},\pi) \leqslant \varepsilon.
\end{align*}
\end{proposition}
To sum up, Proposition~\ref{prop:5} ensures that we can find an input curve which empirical distribution is as close to the target distribution $ \pi$ as we want.

\section{Proof of Proposition~\ref{prop:dual}}
\label{sec:proofp1}
\begin{definition}[indicator function]
Let $\bB\subseteq \R^n$. The indicator of $\bB$ is denoted $\indic_{\bB}$ and defined by:
\begin{align*}
\indic_{\bB}(x)=\left\{\begin{array}{lc}
0 \qquad &\mbox{if } \bx\in \bB \\
+\infty & \mbox{otherwise}
\end{array}
\right.
\end{align*}
\end{definition}

Let us now recall a classical result of convex optimization \cite[P. 195]{hiriart1996convex}:
\begin{proposition}
Let $\bB_\alpha=\{\bx\in \R^n, \|\bx\| \leqslant \alpha\}$. Then the following identity holds:
\begin{align*}
\indic_{\bB_\alpha}(\bx)=\sup_{\by \in \R^n} \ \langle \bx,\by \rangle - \alpha \|\by\|_*.
\end{align*}
\end{proposition}

Now, we can prove Proposition~\ref{prop:dual}.
\begin{align*}
& \min_{\bs\in \bcS\cap \bcA} \frac{1}{2} \|\bs-\bc\|_2^2 \\
& = \min_{\bs\in \bcA} \frac{1}{2} \|\bs-\bc\|_2^2 +  \indic_{\bB_\alpha}(\D \bs)+  \indic_{\bB_\beta}(\DD \bs)\\
& = \min_{\bs\in \bcA} \frac{1}{2} \|\bs-\bc\|_2^2 + \sup_{\bq_1,\bq_2\in \R^{n\cdot d}} \langle \D \bs,\bq_1\rangle - \alpha\|\bq_1\|_* \nonumber \\
& \hspace{.5\linewidth} +\langle \DD \bs,\bq_2\rangle - \beta\|\bq_2\|_\star \\
& = \sup_{\bq_1,\bq_2\in \R^{n\cdot d}}  \min_{\bs\in \bcA} \frac{1}{2} \|\bs-\bc\|_2^2 + \langle \bs,\D^* \bq_1\rangle  +\langle \bs,\DD^* \bq_2\rangle \nonumber \\
& \hspace{.5\linewidth} - \alpha\|\bq_1\|_* - \beta\|\bq_2\|_* \label{eq:dual_relation} 
\end{align*}
The relationship between the primal and dual solutions reads $\displaystyle \bs^* = \argmin_{\bs\in \bcA} \frac{1}{2} \|\bs-\bc\|_2^2 + \langle \bs,\D^* \bq_1^*\rangle  +\langle \bs,\DD^* \bq_2^*\rangle $.
The $\sup$ and the $\min$ can be interverted at the third line, due to standard theorems in convex analysis (see e.g. \cite[Theorem 31.3]{rockafellar1997convex}).

\section{Proof of Propositions~\ref{prop:regularityF} and \ref{prop:defsstar}}
\label{sec:proofp2}

To show Proposition \ref{prop:defsstar}, first remark that 
\begin{align*}
& \argmin_{\bs \in \bcA} \langle \D\bs, \bq_1 \rangle + \langle \DD\bs, \bq_2  \rangle + \frac {1}{2} \|\bs-\bc\|_2^2 \\
&=  \argmin_{\bs \in \bcA} \frac {1}{2} \|\bs-( \bc - \D \bq_1 - \DD^* \bq_2   ) \|_2^2.
\end{align*}
Therefore, $\bs^*(\bq_1,\bq_2)$ is the orthogonal projection of $\bz = \bc - \D \bq_1 - \DD^* \bq_2 $ onto $\bcA$. 
Since $\bcA$ is not empty, $\bA\bA^+\bv=\bv$, and the set $\bcA=\{\bs\in \R^{n\cdot d}, \bA \bs=\bv\}$ can be rewritten as 
\begin{equation*}
\bcA= \bA^+\bv + \mathrm{ker}(\bA).
\end{equation*}
The vector $\bz-\bs^*(\bq_1,\bq_2)$ is orthogonal to $\bcA$, it therefore belongs to $\mathrm{ker}(\bA)^\perp=\mathrm{im}(\bA^*)$. 
Thus $\bs^*(\bq_1,\bq_2) = \bz + \bA^* \blambda$ for some $\blambda$ such that:
\begin{equation*}
\bA (\bz + \bA^* \blambda) = \bv.
\end{equation*}
This leads to $\blambda= (\bA\bA^*)^{-1}(\bv - \bA\bz)$.
We finally get
\begin{equation*}
\bs^*(\bq_1,\bq_2) = \bz + \bA^* (\bA\bA^*)^{-1}(\bv - \bA\bz),
\end{equation*}
ending the proof.

Proposition~\ref{prop:regularityF} results from standard results of convex optimization. See for instance \cite[Theorem 1]{Nesterov2005} or the book \cite{hiriart1996convex}.

\section{Proof of theorem \ref{thm:cvrate}.}
\label{sec:prthm}

The analysis proposed here follows closely ideas proposed in \cite{weiss2009efficient,Boyer14,beck2009gradient,beck2014fast}.
We will need two results. The first one is a duality result from \cite{Boyer14}.
\begin{proposition}\label{prop:distdual}
Let $f:\R^m\to \R\cup\{\infty\}$ and $g:\R^n\to \R\cup \{\infty\}$ denote two closed convex functions. 
Assume that $g$ is $\sigma$-strongly convex \cite{hiriart1996convex} and that $\bA\mathrm{ri}(\mathrm{dom}(f)) \cap \mathrm{ri}(\mathrm{dom}(g))\neq\emptyset$.
Let $\bA\in \R^{m\times n}$ denote a matrix.

Let $p(\bx)=f(\bA\bx)+g(\bx)$ and $d(\by)=-g^*(\bA^*\by) - f^*(\by)$. 
Let $\bx^*$ be the unique minimizer of $p$ and $\by^*$ be any minimizer of $d$. 

Then $g^*$ is differentiable with $\frac{1}{\sigma}$ Lipschitz-continuous gradient. 
Moreover, by letting $\bx(\by) = \nabla g^*(-\bA^* \by)$:
$$
\|\bx(\by)-\bx^*\|_2^2 \leq \frac{2}{\sigma} (d(\by) - d(\by^*)).
$$
\end{proposition}

The second ingredient is the standard convergence rate for accelerated proximal gradient descents given in \cite{beck2009gradient}.
\begin{proposition}
Under the same assumptions as Proposition \ref{prop:distdual}, consider Algorithm~\ref{algo2:max}.
\begin{algorithm}[!ht]
\KwIn{$\bq_0 \in \mathrm{ri}(\mathrm{dom}(f^*)) \cap A\mathrm{ri}(\mathrm{dom}(g^*))$.\\ $n_{it}$}
\textbf{Initialize} 
Set $t=1/L$, with $L=\frac{|||\bA|||^2}{\sigma}$. \\
Set $\by_0=\bq_0$.
\For{$k=1\hdots n_{it}$}{
  $\bq^{(k)}=\prox_{t f^*}(\by^{(k-1)} + t \bA\nabla g^*(-\bA^* \by^{(k-1)}))$\\
  $\by^{(k)}=\bq^{(k)}+\frac{k-1}{k+2}(\bq^{(k)}-\bq^{(k-1)})$\\
}
\label{algo2:max}
\caption{Accelerated proximal gradient descent}
\end{algorithm}

Then $\|\by^{(n_{it})}- \by^*\|_2^2= \mathcal{O}\left( \frac{|||\bA|||^2}{\sigma \cdot n_{it}^2} \right).$
\end{proposition}

To conclude, it suffices to set $g(\bs)=\frac{1}{2}\|\bs-\bc\|_2^2$, $f(\bq_1,\bq_2) = \indic_{\bB_\alpha}(\bq_1) + \indic_{\bB_\alpha}(\bq_2)$ and $\bA=\begin{pmatrix} \D \\ \DD\end{pmatrix}$.
By doing so, the projection problem rewrites $\displaystyle\min_{\bs\in \R^{n\dot d}} p(\bs) = f(A\bs) + g(\bs)$. Its dual problem \eqref{eq:PD} can be rewritten more compactly as $\displaystyle\min_{\bq=(\bq_1,\bq_2)\in \R^{n\dot d} \times \R^{n\dot d}} d(\bq) = g^*(-\bA^*\bq) + f^*(\bq)$. Note that function $g$ is $1$-strongly convex. Therefore, Algorithm \ref{algo2:max} can be used to minimize $d$, ensuring a convergence rate in $\mathcal{O}\left( \frac{1}{k^2} \right)$ on the function values $d(\by^{(k)})$. 
It then suffices to use Proposition~\ref{prop:distdual} to obtain a convergence rate on the distance to the solution $\|\bs^{(k)} - \bs^*\|_2^2$.



\bibliographystyle{ieeetr}
\bibliography{./Biblio/bibenabr,./Biblio/revuedef,./Biblio/revueabr,./Biblio/MyBiblio}

\end{document}